\documentclass[preprint,12pt,authoryear]{elsarticle}

\usepackage{amssymb}
\usepackage{amsfonts}

\usepackage{amsmath,amsthm,amsfonts,amssymb,color}

\usepackage{multicol}

\usepackage{natbib}

\usepackage{multirow} % para las tablas
\usepackage{float}
\usepackage{graphicx}
\usepackage{longtable}  

\usepackage{amssymb}
\usepackage{amsthm}
\newtheorem{theorem}{Theorem}
\newtheorem{corollary}{Corollary}
\newtheorem{lemma}{Lemma}

\def\wH{\widetilde{H}}

\def\wk{\widetilde{K}}
\def\wE{\widetilde{E}}
\def\wN{\widetilde{N}}
\def\wL{\widetilde{L}}
\def\wK{\widetilde{K}}
\newcommand{\bint}{\displaystyle\int}

\usepackage[authoryear]{natbib}
\journal{Journal}

\title{Mathematical modelling of heat transfer in closed electrical contacts and electrical potential field dynamics with Thomson effect}
\author{Targyn A. Nauryz$^{1,2}$, Stanislav N. Kharin$^{1,2}$ Adriana C. Briozzo$^{3,4}$, Julieta Bollati$^{3,4}$}

\address{$^1$International School of Economics, Kazakh-British Technical University, Tole Bi 59, Almaty, Kazakhstan\\
$^2$Institute of Mathematics and Mathematical Modeling, Pushkina 125, Almaty, Kazakhstan\\$^3$CONICET, Argentina\\$^4$Depto de Matem\'{a}tica, F.C.E.,Universidad Austral,\\
Paraguay 1950, S2000FZF Rosario, Argentina}

\begin{document}
\begin{frontmatter}
\begin{abstract}
In this study we develop a mathematical model that describe the behavior of electromagnetic fields and heat transfer in closed electrical contacts  that arises when instantaneous explosion of the micro-asperity which involves vaporization zone and liquid, solid zones where temperature is defined by a generalized heat equation with Thomson effect. This model account for the nonlinear nature of the thermal coefficients and electrical conductivity depended on temperature. Our proposed solutions are based on similarity transformation which allows us to reduce a Stefan-type problem to a system of nonlinear integral equations whose existence of solution is  proved by the fixed point theory in Banach spaces.
\end{abstract}

\end{frontmatter}

\section{Introduction}

Stefan problems are fundamental in understanding phase transition phenomena, particularly in situations involving heat transfer and solidification processes. They were first introduced by Stefan J. in his seminal work in \cite{1}. These problems concern the determination of the moving boundary between phases during the process of solidification or melting.

The classical Stefan problem arises in scenarios where a material undergoes a phase change, such as freezing or melting, subject to certain boundary conditions and physical constraints. One of the key aspects of Stefan problems is the existence of a sharp interface, known as the Stefan interface, which separates the regions of different phases.

Caffarelli and Souganidis in \cite{2} proposed a rate equation approach to tackle Stefan problems, providing a mathematical framework to describe the evolution of the phase boundary. This approach has since been widely utilized in the analysis and numerical simulation of phase-change problems.

%Numerical modeling methodologies have been developed to solve Stefan problems efficiently. Voller and Prakash \cite{3}-\cite{7} introduced fixed-grid numerical techniques for convection-diffusion and conduction-dominant phase-change problems, respectively. These methodologies have been instrumental in understanding complex phase transition phenomena in various engineering applications.

Theoretical investigations into Stefan problems have also contributed significantly to the field. Rubinstein's comprehensive work \cite{4} provides a mathematical foundation for understanding the Stefan problem from a theoretical perspective.  In \cite{5} Alexiades and Solomon delved into the mathematical intricacies of these processes, shedding light on the underlying principles governing phase transitions.

Furthermore, the study of free and moving boundary problems, including Stefan problems, has garnered considerable attention. In \cite{6} Crank provides a comprehensive treatment of such problems, offering valuable insights into their mathematical formulation and solution techniques.

In recent years, researchers have explored various aspects of Stefan problems, including the investigation of similarity solutions. Lipton and Witelski \cite{8} analyze similarity solutions of the one-phase Stefan problem, elucidating important characteristics of the phase-change process.

Stefan problems, which traditionally deal with phase-change phenomena under classical heat conduction assumptions, have seen extensions to encompass more complex physical scenarios. These extensions, often referred to as non-classical Stefan problems, involve variations in thermal coefficients, boundary conditions, or latent heat dependencies, among other factors. The investigation of non-classical Stefan problems has significant implications in various fields, including materials science, engineering, and mathematical physics.

One avenue of research in non-classical Stefan problems involves the consideration of thermal coefficients that vary with temperature or position. In \cite{19}, \cite{20} were  explored Stefan problems for diffusion-convection equations with temperature-dependent thermal coefficients, providing insights into the behavior of phase-change processes under such conditions. Similarly, in \cite{22}, \cite{23} Kumar et al. investigated Stefan problems with variable thermal coefficients, highlighting the impact of these variations on the phase-change dynamics.

Another aspect of non-classical Stefan problems involves incorporating convective boundary conditions or heat flux conditions on fixed faces. The study \cite{20} examined the existence of exact solutions for one-phase Stefan problems with nonlinear thermal coefficients, incorporating Tirskii's method to handle such complexities. Additionally, the paper \cite{21} addressed a one-phase Stefan problem for a non-classical heat equation with a heat flux condition on the fixed face, contributing to the understanding of phase-change phenomena under non-standard boundary conditions.

Moreover, the consideration of latent heat dependencies that vary with position adds another layer of complexity to non-classical Stefan problems. The work \cite{18} provided an explicit solution for a Stefan problem with latent heat depending on the position, utilizing Kummer functions to analyze the phase-change dynamics under such conditions.

Thermal phenomena in electrical apparatus, such as welding, arcing, and bridging, contribute to their failure and are highly complex. These phenomena depend on various factors including current, voltage, contact force, contact material properties, and arc duration \cite{24}, \cite{25}. Experimental investigations usually focus on cumulative probability representations of resulting values \cite{26}-\cite{28}, as direct experimental observation of these processes is often challenging or even impossible due to their extremely short duration.

Hence, mathematical modeling plays a crucial role in understanding the dynamics of such processes, improving the endurance and reliability of contact systems, and predicting and preventing failures in electrical apparatus. However, many existing mathematical models for describing processes like electrical contact welding have limitations. They often make simplifying assumptions, such as adiabatic heating conditions, neglecting phase transformations and dynamic changes in contact radius \cite{29}-\cite{32}. Additionally, important factors like bounce length and duration, time and spatial variations of welding characteristics, and properties of contact materials are often overlooked \cite{32}, \cite{33}.

Efforts have been made in \cite{34}-\cite{39} to address these aspects and the study of electrical contacts involves intricate thermal dynamics influenced by non-linearities in material properties and heat generation mechanisms. Non-linear Stefan problems offer a valuable mathematical framework to model and analyze these complex phenomena, providing insights into heat transfer processes during phase transitions within electrical contacts \cite{40}-\cite{45}. This paper aims to further develop the existing models while also incorporating the Thomson effect.

The Thomson effect refers to the phenomenon where a temperature difference is created across an electrical conductor when an electric current flows through it. This effect occurs due to the interaction between the current-carrying electrons and the lattice structure of the conductor.

In the context of a closure electrical contact after the instantaneous explosion of a micro-asperity, it's important to note that micro-asperities are tiny protrusions or irregularities on the surface of a material. An explosion or sudden release of energy can cause these micro-asperities to rupture or deform.

After such an explosion, the closure electrical contact may be affected in several ways. The intense energy release can lead to the melting or vaporization of the micro-asperities, altering the surface characteristics of the contact. This can potentially disrupt the normal flow of electric current and create temperature variations due to the Thomson effect.

The Thomson effect in this scenario could result in localized heating or cooling at the contact points, depending on the direction of the current flow. This temperature difference might affect the electrical conductivity and overall performance of the closure electrical contact.

It's worth noting that the specific details and consequences of the Thomson effect on a closure electrical contact after an instantaneous explosion of micro-asperity would depend on various factors such as the material properties, energy released, and contact configuration. Therefore, a detailed analysis considering these factors would be necessary to fully understand the implications of the Thomson effect in such a situation.

\begin{figure}\label{fig1}
\centering
\includegraphics[width=6 cm]{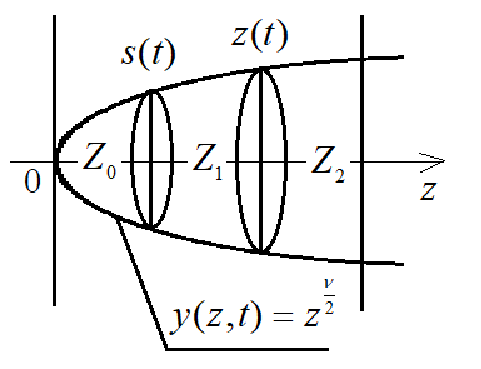}
\caption{Contact zones: $Z_0:(0<z<s(t))$-vaporization zone, $Z_1:(s(t)<z<r(t))$-liquid zone,  $Z_2:(r(t)<z)$-solid zone}
\end{figure} 

In the initial phase of a closed electrical contact, when a micro-asperity undergoes sudden ignition, the contact region comprises both a metallic vaporization zone and a liquid domain, see Figure \ref{fig1}. Modeling the metallic vapor zone, denoted as $Z_0$ with a height range of $0<z<s(t)$, is a complex undertaking. We propose that the temperature within this region decreases linearly from the ionization temperature of the metallic vapor, denoted as $T_{ion}$, which occurs after the explosion at the fixed face $z=0$, to the boiling temperature $T_b$ at the free boundary that separates the vapor and liquid phases. The temperature field within the vapor zone $Z_0$ exhibits a gradual and linear decrease
\begin{equation}\label{eq1}
    T_V(z,t)=\frac{z}{s(t)}(T_b-T_{ion})+T_{ion}, \quad 0\leq z\leq s(t),
\end{equation}
where the following boundary conditions hold
\begin{equation}\label{eq2}
    T_V(0,t)=T_{ion},
\end{equation}
\begin{equation}\label{eq3}
    T_V(s(t),t)=T_b.
\end{equation}
Temperature distribution and electrical potential field of the zones $Z_1$ and $Z_2$ are defined by 
\begin{equation}\label{eq4}
    c(T_1)\gamma(T_1)\dfrac{\partial T_1}{\partial t}=\dfrac{1}{z^{\nu}}\dfrac{\partial}{\partial z}\left[\lambda(T_1)z^{\nu}\dfrac{\partial T_1}{\partial z}\right]+\sigma_{T_1}\dfrac{\partial T_1}{\partial z}\dfrac{\partial\varphi_1}{\partial z}+\dfrac{1}{\rho(T_1)}\left(\dfrac{\partial\varphi_1}{\partial z}\right)^2,
\end{equation}
\begin{equation}\label{eq5}
    \dfrac{1}{z^{\nu}}\dfrac{\partial}{\partial z}\left[\dfrac{1}{\rho(T_1)}z^{\nu}\dfrac{\partial\varphi_1}{\partial z}\right]=0,\quad s(t)<z<r(t),\quad t>0,\quad 0<\nu<1,
\end{equation}
\begin{equation}\label{eq6}
    c(T_2)\gamma(T_2)\dfrac{\partial T_2}{\partial t}=\dfrac{1}{z^{\nu}}\dfrac{\partial}{\partial z}\left[\lambda(T_2)z^{\nu}\dfrac{\partial T_2}{\partial z}\right]+\sigma_{T_2}\dfrac{\partial T_2}{\partial z}\dfrac{\partial\varphi_2}{\partial z}+\dfrac{1}{\rho(T_2)}\left(\dfrac{\partial\varphi_2}{\partial z}\right)^2,
\end{equation}
\begin{equation}\label{eq7}
    \dfrac{1}{z^{\nu}}\dfrac{\partial}{\partial z}\left[\dfrac{1}{\rho(T_2)}z^{\nu}\dfrac{\partial\varphi_2}{\partial z}\right]=0,\quad r(t)<z,\quad t>0,\quad 0<\nu<1,
\end{equation}
\begin{equation}
    T_1(s(t),t)=T_b, \quad t>0,
\end{equation}
\begin{equation}\label{eq8}
    -\lambda\left(T_1(s(t),t)\right)\dfrac{\partial T_1}{\partial z}\Bigg|_{z=s(t)}=\dfrac{Q_0e^{-s_0^2}}{2a\sqrt{\pi t}},\quad t>0,
\end{equation}
\begin{equation}\label{eq9}
    \varphi_1(s(t),t)=0,\quad t>0,
\end{equation}
\begin{equation}\label{eq10}
    T_1(r(t),t)=T_2(r(t),t)=T_m>0,\quad t>0,
\end{equation}
\begin{equation}\label{eq11}
    \varphi_1(r(t),t)=\varphi_2(r(t),t),\quad t>0,
\end{equation}
\begin{equation}\label{eq12}
    -\lambda\left(T_1(r(t),t)\right)\dfrac{\partial T_1}{\partial z}\Bigg|_{z=r(t)}+\lambda\left(T_2(r(t),t)\right)\dfrac{\partial T_2}{\partial z}\Bigg|_{z=r(t)}=l_m\gamma_m\dfrac{dr}{dt},\quad t>0,
\end{equation}
\begin{equation}\label{eq13}
    \dfrac{1}{\rho(T_1(r(t),t))}\dfrac{\partial\varphi_1}{\partial z}\Bigg|_{z=r(t)}=\dfrac{1}{\rho(T_2(r(t),t))}\dfrac{\partial\varphi_2}{\partial z}\Bigg|_{z=r(t)},\quad t>0,
\end{equation}
\begin{equation}\label{eq14}
    T_2(+\infty,t)=0,\quad t>0,
\end{equation}
\begin{equation}\label{eq15}
    \varphi_2(+\infty,t)=\dfrac{U_c}{2},\quad t>0,
\end{equation}
\begin{equation}\label{eq16}
   T_2(z,0)=\varphi_2(z,0)=0,\quad z>0,\qquad s(0)=r(0)=0,\quad 
\end{equation}
where $T_1,\;T_2$ and $\varphi_1,\;\varphi_2$ are temperatures and electrical potential fields for liquid and solid zones, $c(T_i)$,$\gamma(T_i)$ and $\lambda(T_i)$ are specific heat, density and thermal conductivity which depend on the temperature, $\sigma_{T_i}$ is the Thomson coefficient, $\rho(T_i)$ is the electrical resistivity, $Q_0>0$ is the power of the heat flux, $T_m$ is the melting temperature, $U_c$ is the contact voltage, $s(t)$ and $r(t)$ are locations of the boiling and melting interfaces.

This paper is structured as follows. In Section 2, we use the similarity transformation to obtain an equivalent system of coupled integral equations for the problem \eqref{eq4}-\eqref{eq16}. In Section 3, we define proper spaces in order to apply the fixed point Banach theorem to prove 
the existence of solution to the system of coupled integral equations. 

The contribution of the problem addressed in our paper holds significant implications for electrical engineering. By developing a mathematical model that captures the behavior of electromagnetic fields and heat transfer in closed electrical contacts, particularly during instantaneous micro-asperity explosions, we offer valuable insights into the complex dynamics of these systems.

Our model accounts for the non-linear nature of thermal coefficients and temperature-dependent electrical conductivity, factors that are crucial in accurately representing real-world scenarios. By considering vaporization, liquid, and solid zones within the contact, we provide a comprehensive framework for analyzing the thermal and electromagnetic effects associated with such phenomena.

Furthermore, our approach, which utilizes similarity transformations to reduce the Stefan-type problem to a system of nonlinear integral equations, offers practical methodologies for analyzing and predicting the behavior of electrical contacts under extreme conditions. The rigorous establishment of the validity of this approach through discussions and proofs supported by fixed point theory in Banach space enhances the reliability and applicability of our proposed solutions.

Overall, our contribution bridges theoretical insights with practical applications, providing electrical engineers with tools and methodologies to better understand, design, and optimize the performance and reliability of electrical contacts in various operating conditions.

\section{Integral formulation}
In this section, taking into account that the problem \eqref{eq4}-\eqref{eq16} can be thought as a Stefan-type problem, 
we look for similarity type solutions that depend on the similarity variable 
$$
\eta=\dfrac{z}{2a\sqrt{t}},
$$
with $a=\sqrt{\tfrac{\lambda_0}{\rho_0 c_0}}$ where $\lambda_0$, $\rho_0$ and $c_0$ are  reference thermal coefficients.

We propose the following transformation
\begin{equation}\label{eq24}
 f_i(\eta)=\frac{T_i(z,t)-T_m}{T_m},\qquad \phi_i(\eta)= \varphi_i(z,t),\quad ,\quad i=1,2.
\end{equation}
According to this transformation, the location of the boiling and melting fronts are given by
\begin{equation}\label{eq25}
    s(t)=2as_0\sqrt{t},\qquad \qquad  r(t)=2ar_0\sqrt{t},
\end{equation}
where $s_0$ and $r_0$ must be determined as part of the solution.

Therefore, the  problem \eqref{eq4}-\eqref{eq16} can be rewritten as the following form
\begin{equation}\label{eq26}
    \left[L(f_i)\eta^{\nu}f_i'\right]'+2a\eta^{\nu+1}N(f_i)f_i'+\dfrac{\sigma_{f_i}}{c_0\gamma_0a}\eta^{\nu}f_i'\phi_i'+\dfrac{\eta^{\nu}}{c_0\gamma_0T_m a K(f_i)}\left(\phi_i'\right)^2=0,
\end{equation}
\begin{equation}\label{eq27}
    \left[\dfrac{1}{K(f_i)}\eta^{\nu}\phi_i'\right]'=0,
\end{equation}
\begin{equation*}
    i=1:\quad s_0<\eta<r_0,\qquad \quad i=2:\quad \eta>r_0,
\end{equation*}
\begin{equation}\label{s0}
    f_1(s_0)=B,
\end{equation}
\begin{equation}\label{eq28}
    L(f_1(s_0))f_1'(s_0)=-Qe^{-s_0^2},
\end{equation}
\begin{equation}\label{eq29}
    \phi_1(s_0)=0,
\end{equation}
\begin{equation}\label{eq30}
    f_1(r_0)=f_2(r_0)=0,
\end{equation}
\begin{equation}\label{eq31}
    \phi_1(r_0)=\phi_2(r_0),
\end{equation}
\begin{equation}\label{eq32}
    -L(f_1(r_0))f_1'(r_0)=-L(f_2(r_0))f_2'(r_0)+M r_0,
\end{equation}
\begin{equation}\label{eq33}
    \dfrac{1}{K(f_1(r_0))}\phi_1'(r_0)=\dfrac{1}{K(f_2(r_0))}\phi_2'(r_0),
\end{equation}
\begin{equation}\label{eq34}
    f_2(+\infty)=-1,
\end{equation}
\begin{equation}\label{eq35}
    \phi_2(+\infty)=\dfrac{U_c}{2},
\end{equation}
where 
\begin{equation}\label{eq36}
 B=\frac{T_b-T_m}{T_m},\qquad   Q=\dfrac{Q_0}{\lambda_0T_m\sqrt{\pi}}>0,\quad M=\dfrac{2l_m\gamma_m a^2}{\lambda_0T_m}>0
\end{equation}
and for $i=1,2$:
\begin{equation}\label{N}
    N(f_i)=\frac{c(f_iT_m+T_m)\gamma(f_iT_m+T_m)}{c_0\gamma_0}, 
\end{equation}
\begin{equation}\label{L}
    L(f_i)=\frac{\lambda(f_iT_m+T_m)}{\lambda_0}, 
\end{equation}
\begin{equation}\label{K}
    K(f_i)=\rho(f_iT_m+T_m), 
\end{equation}
\begin{equation}\label{sigma}
    \sigma_{f_i}=\sigma_{T_i}, 
\end{equation}
From \eqref{eq27}, \eqref{eq29}, \eqref{eq31}, \eqref{eq33} and \eqref{eq35}, we obtain the solution for electrical potential field for liquid and solid zones explicitly in function of $f_1,f_2,s_0$ and $r_0$ as
\begin{equation}\label{phi1}
    \phi_1(\eta,s_0,r_0,f_1,f_2)= \dfrac{U_c F_1(\eta,s_0,f_1)}{2H(r_0,s_0,f_1,f_2)},\quad s_0\leq \eta \leq r_0,
\end{equation}
\begin{equation}\label{phi2}
    \phi_2(\eta,s_0,r_0,f_1,f_2)=\dfrac{U_c\left(F_1(r_0,s_0,f_1)+F_2(\eta,r_0,f_2)\right)}{2H(r_0,s_0,f_1,f_2)},\quad \eta\geq r_0,
\end{equation}
where 
\begin{equation}\label{F1}
    F_1(\eta,s_0,f_1)=\int\limits_{s_0}^{\eta}\dfrac{K(f_1(v))}{v^{\nu}}dv,\quad s_0\leq\eta\leq r_0,
\end{equation}
\begin{equation}\label{F2}
    F_2(\eta,r_0,f_2)=\bint\limits_{r_0}^{\eta}\dfrac{K(f_2(v))}{v^{\nu}}dv,\quad \eta\geq r_0,
\end{equation}and
\begin{equation}\label{H}
H(r_0,s_0,f_1,f_2)=F_1(r_0,s_0,f_1)+F_2(+\infty,r_0,f_2).
\end{equation}
In addition, from $\eqref{eq26}$, $\eqref{eq28}$ and $\eqref{eq30}$, we get
    \[f_1(\eta)=s_0^{\nu}Q\exp(-s_0^{2})\left[\Phi_1(r_0,s_0,f_1,f_2)-\Phi_1(\eta,s_0,f_1,f_2)\right]
    \]
  \begin{equation}\label{eq41}
  + \frac{ D_1^*}{H^{2}(r_0,s_0,f_1,f_2)}\left[{G_1(r_0,s_0,f_1,f_2)}-G_1(\eta,s_0,f_1,f_2)\right],\quad s_0\leq\eta\leq r_0, 
\end{equation}
and from $\eqref{eq26}$, $\eqref{eq30}$ and \eqref{eq34} we get
 \[
    f_2(\eta)=\left[\frac{D_2^*}{H^{2}(r_0,s_0,f_1,f_2)}G_2(+\infty,r_0,f_1,f_2)-1\right]\frac{\Phi_2(\eta,r_0,f_1,f_2)}{\Phi_2(+\infty,r_0,f_1,f_2)}
  \] 
 \begin{equation}\label{eq42} 
 -\frac{D_2^*}{H^{2}(r_0,s_0,f_1,f_2)}G_2(\eta,r_0,f_1,f_2), \quad \eta\geq r_0.
  \end{equation}
  Moreover, from conditions \eqref{s0} and \eqref{eq32} we obtain the following equations
  \begin{equation}\label{Ec1-s0,r0}
  s_0^{\nu}Q\exp(-s_0^{2})\Phi_1(r_0,s_0,f_1,f_2)+ \frac{ D_1^*}{H^{2}(r_0,s_0,f_1,f_2)}{G_1(r_0,s_0,f_1,f_2)}=B,
\end{equation}
and
   \[ E_1(r_0,s_0,f_1,f_2)\left[Q\exp(-s_0^{2})s_0^{\nu}+\frac{ D_1^*}{H^{2}(r_0,s_0,f_1,f_2)}H_1(r_0,s_0,f_1,f_2)\right]
   \]
   \begin{equation}\label{Ec2-s0,r0}-\frac{1}{\Phi_2(+\infty,r_0,f_1,f_2)}\left[1-\frac{ D_2^*}{H^{2}(r_0,s_0,f_1,f_2)}G_2(+\infty,r_0,f_1,f_2)\right]=M r_0^{\nu+1}
\end{equation}  
where
\begin{equation}\label{eq43}
\Phi_1(\eta,s_0,f_1,f_2)=\bint\limits_{s_0}^{\eta}\dfrac{E_1(v,s_0,f_1,f_2)}{L(f_1(v))v^{\nu}}dv,\quad s_0\leq\eta\leq r_0,
\end{equation}
\begin{equation}\label{eq44}
\Phi_2(\eta,r_0,f_1,f_2)=\bint\limits_{r_0}^{\eta}\dfrac{E_2(v,r_0,f_1,f_2)}{L(f_2(v))v^{\nu}}dv,\quad \eta\geq r_0,
\end{equation}
\begin{equation}\label{eq45}
    G_1(\eta,s_0,f_1,f_2)=\bint\limits_{s_0}^{\eta}\dfrac{E_1(v,s_0,f_1,f_2)}{L(f_1(v))v^{\nu}}H_1(v,r_0,f_1,f_2)dv,\quad s_0\leq\eta\leq r_0,
\end{equation}
\begin{equation}\label{eq46}
G_2(\eta,r_0,f_1,f_2)=\bint\limits_{r_0}^{\eta}\dfrac{E_2(v,s_0,f_1,f_2)}{L(f_2(v))v^{\nu}}H_2(v,r_0,f_1,f_2)dv\quad \eta\geq r_0
\end{equation}
\begin{equation}\label{defH1}
H_1(\eta,s_0,f_1,f_2)=\bint\limits_{s_0}^{\eta}\dfrac{K(f_1(v))}{v^{\nu}E_1(v,s_0,f_1,f_2)}dv,\quad s_0\leq\eta\leq r_0,
\end{equation}
\begin{equation}
H_2(\eta,r_0,f_1,f_2)=\bint\limits_{r_0}^{\eta}\dfrac{K(f_2(v))}{v^{\nu}E_2(v,r_0,f_1,f_2)}dv\quad \eta\geq r_0
\end{equation}
\begin{equation}\label{defE1}
    E_1(\eta,s_0,f_1,f_2)=\exp\left(-\bint\limits_{s_0}^{\eta}\left[2av\tfrac{N(f_1(v))}{L(f_1(v))}+\tfrac{D_1}{H(r_0,s_0,f_1,f_2)}\tfrac{K(f_1(v))}{L(f_1(v))v^{\nu}}\right]dv\right),\quad s_0\leq\eta\leq r_0,
\end{equation}
\begin{equation}\label{eq48}
    E_2(\eta,r_0,f_1, f_2)=\exp\left(-\bint\limits_{r_0}^{\eta}\left[2av\tfrac{N(f_2(v))}{L(f_2(v))}+\tfrac{D_2}{H(r_0,s_0,f_1,f_2)}\tfrac{K(f_2(v))}{L(f_2(v))v^{\nu}}\right]dv\right),\quad \eta\geq r_0,
\end{equation}
and the coefficients $D_i$ and   $D_i^*$ for $i=1,2$ are given by
\begin{equation}\label{D1D2}
    D_i=\dfrac{\sigma_{f_i}U_c}{2c_0\gamma_0a},\qquad \qquad D_i^*=\frac{U_c D_i}{2}.
\end{equation}

In conclusion, to find a similarity solution to the problem  \eqref{eq4}-\eqref{eq16} is equivalent to obtain $f_1$, $f_2$, $s_0$ and $r_0$ such that  \eqref{eq41}, \eqref{eq42}, \eqref{Ec1-s0,r0} and \eqref{Ec2-s0,r0} hold. Notice that the electric potential fields $\phi_1$ and $\phi_2$ are explicitly given by \eqref{phi1} and \eqref{phi2} in function of  $f_1$, $f_2$, $s_0$ and $r_0$. 

In the next section, to address the existence and uniqueness of solutions, we employ a rigorous analytical approach. We leverage similarity transformations to reduce the problem to a set of ordinary differential equations, facilitating a more tractable analysis. Additionally, we draw upon fixed point theory in Banach space to establish the validity of our proposed solutions.

Through thorough discussions and rigorous proofs, we demonstrate the soundness of our approach and the robustness of our solutions. This section not only provides theoretical insights into the existence and uniqueness of solutions within the context of electrical contact phenomena but also offers practical methodologies for analyzing and predicting the behavior of such systems under extreme conditions.

Overall, our analysis in this section contributes to a deeper understanding of the mathematical underpinnings of electromagnetic fields and heat transfer in closed electrical contacts. By establishing the existence and uniqueness of solutions, we lay a solid foundation for further exploration and optimization of electrical contact systems in various engineering applications.

\section{Existence of solution}
In order to prove the existence and uniqueness of solution $f_1,f_2$ to equations \eqref{eq41} and \eqref{eq42}, we assume fixed positive constants $0<s_0<r_0$ and consider the Banach space 
\begin{equation}\label{mathcal C}
\mathcal{C}=C[s_0,r_0]\times C_b[r_0,+\infty)    
\end{equation}
endowed with the norm 
$$||\vec{f}||=\Vert (f_1,f_2)\Vert=\max\left\lbrace ||f_1||_{C[s_0,r_0]},||f_2||_{C_b[r_0,+\infty)} \right\rbrace$$
 where $C[s_0,r_0]$ denotes the space of continuous functions defined on the interval $[s_0,r_0]$ and $C_b[r_0,+\infty)$ represents the space of continuous and bounded functions on the interval $[r_0,+\infty).$
We define the closed subset $\mathcal{M}$ of $C_b[r_0,+\infty)$ by
$$\mathcal{M}=\lbrace f_2\in C_b[r_0,+\infty):  f_2(r_0)=0, f_2(+\infty)=-1 \rbrace.$$
We consider the operator $\Psi$ on $\mathcal{K}=C[s_0,r_0]\times \mathcal{M}$ given by
\begin{equation}\label{Psi}
    \Psi(\vec{f})=(V_1(\vec{f}),V_2(\vec{f})),
\end{equation}
 where  $V_1(\vec{f})$, $V_2(\vec{f})$ are defined by
\begin{equation}\label{V1}
    \begin{array}{llll}
        &V_1(\vec{f})(\eta)  =s_0^{\nu}Q\exp(-s_0^{2})\left[\Phi_1(r_0,s_0,f_1,f_2)-\Phi_1(\eta,s_0,f_1,f_2)\right] \\ \\
  &+ \frac{ D_1^*}{H^{2}(r_0,s_0,f_1,f_2)}\Big[{G_1(r_0,s_0,f_1,f_2)}-G_1(\eta,s_0,f_1,f_2)\Big],\quad s_0\leq\eta\leq r_0, 
  \end{array}
\end{equation}

 \begin{equation}
 \begin{array}{llll}
    &V_2(\vec{f})(\eta)=\Big[\frac{D_2^*}{H^{2}(r_0,s_0,f_1,f_2)}G_2(+\infty,r_0,f_1,f_2)-1\Big]\frac{\Phi_2(\eta,r_0,f_1,f_2)}{\Phi_2(+\infty,r_0,f_1,f_2)}\\ \\
 &-\frac{D_2^*}{H^{2}(r_0,s_0,f_1,f_2)}G_2(\eta,r_0,f_1,f_2), \quad \eta\geq r_0.
  \end{array}
  \end{equation}
  Notice that  solving the system of equations \eqref{eq41} and \eqref{eq42} is equivalent to obtain a fixed point  to the operator $\Psi$.

Taking into account that $\mathcal{K}$ is a closed subset of $\mathcal{C}$ we will prove that $\Psi(\mathcal{K})\subset \mathcal{K}$ and $\Psi$ is a contraction mapping in order to apply the fixed point Banach theorem.

For this purpose we will assume that there exists  positive coefficients $\mu$, $L_{im},\;L_{iM},\;N_{im}$ and $N_{iM}$, $\widetilde{L}_i$ , $\widetilde{N}_i$ and $\widetilde{K}_i$ for $i=1,2$ such that

\begin{enumerate}
\item[(A1)] for each $f_1\in C[s_0,r_0]: s_0\leq v\leq r_0 $
    \begin{equation}\label{cotaLf1}
        L_{1m}\eta^{\mu}\leq L(f_1)(\eta)\leq L_{1M}\eta^{\mu},
    \end{equation}
    \begin{equation}\label{cotaNf1}
        N_{1m}\eta^{-\mu}\leq N(f_1)(\eta)\leq N_{1M}\eta^{-\mu},
    \end{equation}
    \begin{equation}\label{cotaKf1}
        K_{1m}\eta^{-\mu}\leq K(f_1)(\eta)\leq K_{1M}\eta^{-\mu},
    \end{equation}
 \item[(A2)]  for each $f_2\in \mathcal{M},\;\eta\geq r_0: $
    \begin{equation}\label{cotaLf2}
        L_{2m} \eta^{\mu}\leq L(f_2)(\eta)\leq L_{2M}\eta^{\mu} ,
    \end{equation}
    \begin{equation}\label{cotaNf2}
        N_{2m} \eta^{-\mu}\leq N(f_2)(\eta)\leq N_{2M}\eta^{-\mu} ,
    \end{equation}
    \begin{equation}\label{cotaKf2}
       K_{2m} \eta^{-\mu}\leq K(f_2)(\eta)\leq K_{2M}\eta^{-\mu} ,
    \end{equation}

    \item[(A3)] for each $f_1, g_1\in C[s_0,r_0], s_0\leq \eta\leq r_0 $:
    \begin{equation}\label{LipschitzL1}
        |L(f_1(\eta))-L(g_1(\eta))|\leq \widetilde{L}_1||f_1-g_1||,
    \end{equation}
    \begin{equation}\label{LipschitzN1}
        |N(f_1(\eta))-N(g_1(\eta))|\leq \widetilde{N}_1||f_1-g_1||,
    \end{equation}
\begin{equation}\label{LipschitzK1}
        |K(f_1(\eta))-K(g_1(\eta))|\leq \widetilde{K}_1 \eta^{-\mu}||f_1-g_1||,
    \end{equation}

     \item[(A4)] for each $f_2, g_2\in \mathcal{M}, \eta\geq r_0$ :
    \begin{equation}\label{LipschitzL2}
        |L(f_2(\eta))-L(g_2(\eta))|\leq \widetilde{L}_2||f_2-g_2||,
    \end{equation}
    \begin{equation}\label{LipschitzN2}
        |N(f_2(\eta))-N(g_2(\eta))|\leq \widetilde{N}_2||f_2-g_2||,
    \end{equation}
\begin{equation}\label{LipschitzK2}
        |K(f_2(\eta))-K(g_2(\eta))|\leq \widetilde{K}_2 \eta^{-\mu}||f_2-g_2||,
    \end{equation}
    
    \item[(A5)]   $\mu>2$.
\end{enumerate}

From now on, hypothesis (A1)-(A5) will be assumed throughout the paper.
\bigskip

We will present preliminary results that will be useful to prove the existence and uniqueness of a fixed point to the operator $\Psi$.

\begin{lemma}\label{lema-cotaH} For every $\vec{f}=(f_1,f_2)$, $ \vec{g}=(g_1,g_2)\in\mathcal{K}$, the following inequalities hold: 
    \begin{equation}\label{cotaHinf}
         H(r_0,s_0,f_1,f_2)\geq H_{inf}(r_0,s_0),
    \end{equation}
    \begin{equation}\label{cotaHsup}
        H(r_0,s_0,f_1,f_2)\leq H_{sup}(r_0,s_0),
    \end{equation}
    \begin{equation}\label{difH}
        |H(r_0,s_0,f_1,f_2)-H(r_0,s_0,g_1,g_2)|\leq \wH(r_0,s_0) ||\vec{f}-\vec{g}||,
    \end{equation}
  where
    \begin{equation}\label{Hinf}
        H_{inf}(r_0,s_0):=\tfrac{K_{1m} }{\mu+\nu-1} \left(\tfrac{1}{s_0^{\mu+\nu-1}}-\tfrac{1}{r_0^{\mu+\nu-1}} \right)
    \end{equation}
    \begin{equation}\label{Hsup}
        H_{sup}(r_0,s_0):=\tfrac{1 }{\mu+\nu-1} \left(\tfrac{K_{1M}}{s_0^{\mu+\nu-1}}+\tfrac{K_{2M}}{r_0^{\mu+\nu-1}} \right)
    \end{equation}
    \begin{equation}\label{Hraya}
        \wH(r_0,s_0):=\tfrac{1 }{\mu+\nu-1} \left(\tfrac{\wK_1}{s_0^{\mu+\nu-1}}+\tfrac{\wK_2}{r_0^{\mu+\nu-1}} \right) 
    \end{equation}
\end{lemma}

\begin{proof}
    Taking into account the definition of $H$ given by \eqref{H} and assumptions \eqref{cotaKf1} and \eqref{cotaKf2} we have
$$\begin{array}{ll}
H(r_0,s_0,f_1,f_2)\geq \tfrac{K_{1m} }{\mu+\nu-1} \left(\tfrac{1}{s_0^{\mu+\nu-1}}-\tfrac{1}{r_0^{\mu+\nu-1}} \right)+\tfrac{K_{2m}}{\mu+\nu-1} \tfrac{1}{r_0^{\mu+\nu-1}}\\ \\
\geq \tfrac{K_{1m} }{\mu+\nu-1} \left(\tfrac{1}{s_0^{\mu+\nu-1}}-\tfrac{1}{r_0^{\mu+\nu-1}} \right),
\end{array}$$
and then we get \eqref{cotaHinf}.
Inequality \eqref{cotaHsup} follows analogously. In addition, taking into account assumptions \eqref{LipschitzK1} and \eqref{LipschitzK2} we get
\begin{equation*}
\begin{array}{ll}
       & |H(r_0,s_0,f_1,f_2)-H(r_0,s_0,g_1,g_2)|\\[0.35cm]
        & \leq \left(\wk_1 \bint_{s_0}^{r_0} \tfrac{1}{v^{\mu+\nu}} dv+\wk_2 \bint_{r_0}^{+\infty} \tfrac{1}{v^{\mu+\nu}}  dv \right)||\vec{f}-\vec{g}|| \\ \\
       &\leq \tfrac{1 }{\mu+\nu-1} \left(\wK_1 \left(\tfrac{1}{s_0^{\mu+\nu-1}}-\tfrac{1}{r_0^{\mu+\nu-1}}\right)+\tfrac{\wK_2}{r_0^{\mu+\nu-1}} \right) ||\vec{f}-\vec{g}|| \\ \\
       &\leq \tfrac{1 }{\mu+\nu-1} \left(\tfrac{\wK_1}{s_0^{\mu+\nu-1}}+\tfrac{\wK_2}{r_0^{\mu+\nu-1}} \right)||\vec{f}-\vec{g}||,
        \end{array}
    \end{equation*}
as a consequence, inequality \eqref{difH} holds.
\end{proof}

\begin{lemma}\label{lema-cotaFase1} For every $\vec{f}=(f_1,f_2)$,  $\vec{g}=(g_1,g_2)\in\mathcal{K}$, the following inequalities hold: 
\begin{enumerate}
    
     \item  
    \begin{equation}\label{cotaE1inf}
        E_1(\eta,s_0,f_1,f_2)\geq E_{1 inf}(r_0,s_0)
    \end{equation}
    \begin{equation}\label{cotaE1sup}
        E_1(\eta,s_0,f_1,f_2)\leq 1
    \end{equation}
    \begin{equation}\label{difE1}
        |E_1(\eta,s_0,f_1,f_2)-E_1(\eta,s_0,g_1,g_2)|\leq \wE_1(r_0,s_0) ||\vec{f}-\vec{g}||
    \end{equation}
where
 \begin{equation}\label{defE1inf}
 E_{1inf}(r_0,s_0):=\exp\left(-\left[  a  \tfrac{N_{1M} }{L_{1m} (\mu-1)} \tfrac{1}{ s_0^{2\mu-2}}+\tfrac{D_1 K_{1M} }{H_{inf}(r_0,s_0) L_{1m}(2\mu+\nu-1)} \tfrac{1}{s_0^{2\mu+\nu-1}}\right]\right)
 \end{equation}
 \begin{equation}\label{E1raya}
 \begin{array}{l}
 \widetilde{E}_1(r_0,s_0):= 2a \left[ \tfrac{ \wN_1 }{L_{1m} (\mu-2)}\tfrac{1}{s_0^{\mu-2}}+\tfrac{ N_{1M} \wL_1}{L_{1m}^2 (3\mu-2)} \tfrac{1}{s_0^{3\mu-2}}  \right]\\ \\
 +D_1\Big(  \tfrac{\wK_1}{H_{inf}(r_0,s_0) L_{1m} (2\mu+\nu-1)}\tfrac{1}{s_0^{2\mu+\nu-1}}\\ \\+\tfrac{K_{1M}}{H_{inf}(r_0,s_0) L_{1m}} \Big( \tfrac{\widetilde{H}(r_0,s_0)}{H_{inf}(r_0,s_0) (2\mu+\nu-1)}\tfrac{1}{s_0^{2\mu+\nu-1}}+ \tfrac{\widetilde{L}_1}{ L_{1m} (3\mu+\nu-1)}\tfrac{1}{s_0^{3\mu+\nu-1}} \Big)\Big)
 \end{array}
 \end{equation}

\item 
%\begin{equation}\label{cotaPhi1inf}
%         \Phi_1(s_0,r_0,f_1,f_2)\geq \Phi_{1 inf}(s_0, r_0)
%    \end{equation}
%    \begin{equation}\label{cotaPhi1sup}
%      \Phi_1(s_0,r_0,f_1,f_2)\leq \Phi_{1sup}(s_0)
%    \end{equation}
    \begin{equation}\label{Resta Phi 1}
|\Phi_1(\eta,s_0,f_1,f_2)-\Phi_1(\eta,s_0,g_1,g_2)|\leq \widetilde{\Phi}_1(r_0,s_0)||\vec{f}-\vec{g}||
\end{equation}

  where  
  \begin{equation}\label{Phi1raya}
 \widetilde{\Phi}_1(r_0,s_0):=\tfrac{\wE_1(r_0,s_0) }{L_{1m}(\mu+\nu-1)}\frac{1}{s_0^{\nu+\mu-1}}+\tfrac{\widetilde{L}_1}{L_{1m}^{2}(2\mu+\nu-1)}\frac{1}{s_0^{\nu+2\mu-1}}
 \end{equation}
  %\begin{equation}\label{Phi1inf}
%\Phi_{1 inf}(s_0,r_0):=\frac{E_{1inf}}{L_{1M}(\mu+\nu+1)}\bigg(\frac{1}{s_0^{\mu+\nu+1}}-\frac{1}{r_0^{\mu+\nu+1}}\bigg)
%  \end{equation}  
%    \begin{equation}\label{Phi1sup}
%\Phi_{1 sup}(s_0):= \frac{1}{L_{1m}(\mu+\nu+1)}\frac{1}{s_0^{\mu+\nu+1}}
%\end{equation}

\item 
\begin{equation}\label{cotaH1inf}
         H_1(\eta,s_0,f_1,f_2)\geq H_{1 inf}(\eta, s_0)
    \end{equation}
    \begin{equation}\label{cotaH1sup}
      H_1(\eta,s_0,f_1,f_2)\leq H_{1sup}(r_0,s_0)
    \end{equation}
    \begin{equation}\label{Resta H1}
|H_1(\eta,s_0,f_1,f_2)-H_1(\eta,s_0,g_1,g_2)|\leq \widetilde{H}_1(r_0,s_0)||\vec{f}-\vec{g}||
\end{equation}
 where  
 \begin{equation}\label{H1inf}
H_{1 inf}(\eta,s_0):= \tfrac{K_{1m}}{(\mu+\nu-1)} \left(\tfrac{1}{s_0^{\mu+\nu-1}}-\tfrac{1}{\eta^{\mu+\nu-1}}\right)
 \end{equation}  
    \begin{equation}\label{H1sup}
H_{1 sup}(r_0,s_0):= \tfrac{K_{1M}}{E_{1inf}(r_0,s_0)} \tfrac{1}{(\mu+\nu-1)} \tfrac{1}{s_0^{\mu+\nu-1}}
\end{equation}
 \begin{equation}\label{H1raya}
 \widetilde{H}_1(r_0,s_0):= \left( \wK_1+ \tfrac{K_{1M} \wE_1(r_0,s_0)}{E_{1inf}(r_0,s_0)}\right) \tfrac{1}{E_{1inf}(r_0,s_0) (\mu+\nu-1)} \tfrac{1}{s_0^{\mu+\nu-1}}
 \end{equation}
\item 
\begin{equation}\label{cotaG1inf}
         G_1(\eta,s_0,f_1,f_2)\geq G_{1 inf}(\eta, r_0,s_0)
    \end{equation}
    \begin{equation}\label{cotaG1sup}
      G_1(\eta,s_0,f_1,f_2)\leq G_{1sup}(r_0,s_0)
    \end{equation}
    \begin{equation}\label{Resta G 1}
|G_1(\eta,s_0,f_1,f_2)-G_1(\eta,s_0,g_1,g_2)|\leq \widetilde{G}_1(r_0,s_0)||\vec{f}-\vec{g}||
\end{equation}
 where 
 \begin{equation}\label{G1inf}
G_{1 inf}(\eta,r_0,s_0):= \tfrac{K_{1m} E_{1inf}(r_0,s_0) }{2 L_{1M} (\mu+\nu-1)^2} \left(\tfrac{1}{s_0^{\mu+\nu-1}}-\tfrac{1}{\eta^{\mu+\nu-1}} \right)^2,
  \end{equation}  
    \begin{equation}\label{G1sup}
G_{1 sup}(r_0,s_0):= \tfrac{H_{1sup}(r_0,s_0)}{L_{1m}} \tfrac{1}{(\mu+\nu-1)}\tfrac{1}{s_0^{\mu+\nu-1}},
\end{equation}
 \begin{equation}\label{G 1 raya}
 \widetilde{G}_1(r_0,s_0):= H_{1sup}(r_0,s_0) \widetilde{\Phi}_1(r_0,s_0)+ \tfrac{\widetilde{H}_1(r_0,s_0)}{L_{1m}}  \tfrac{1}{(\mu+\nu-1)}\tfrac{1}{s_0^{\mu+\nu-1}}.
 \end{equation}
\end{enumerate}
\end{lemma}

\begin{proof}
    From the definition of $E_1$ given by \eqref{defE1}, assumptions \eqref{cotaLf1}-\eqref{cotaKf1} and inequality \eqref{cotaHinf} we obtain that
    $$\begin{array}{l}
    \bint_{s_0}^{\eta} 2a v \tfrac{N(f_1(v))}{L(f_1(v)}+\tfrac{D_1 }{H(r_0,s_0,f_1,f_2)} \tfrac{K(f_1(v))}{L(f_1(v)) v^{\nu}} dv\\ \\
    \leq\bint_{s_0}^{\eta} 2a  \tfrac{N_{1M} }{L_{1m}} \tfrac{1}{v^{2\mu-1}}+\tfrac{D_1 K_{1M} }{H_{inf}(r_0,s_0) L_{1m}} \tfrac{1}{v^{2\mu+\nu}} dv \\ \\
 \leq  2a  \tfrac{N_{1M} }{L_{1m} (2\mu-2)} \tfrac{1}{ s_0^{2\mu-2}}+\tfrac{D_1 K_{1M} }{H_{inf}(r_0,s_0) L_{1m}(2\mu+\nu-1)} \tfrac{1}{s_0^{2\mu+\nu-1}}. \\ \\
    \end{array}
    $$
As a consequence it results that
   $ E_1(\eta,s_0,f_1,f_2)\geq  E_{1inf}(r_0,s_0)$ with $E_{1inf}(r_0,s_0)$ given by \eqref{defE1inf}.

   In addition, as $E_1$ is a negative exponential function it follows that $E_1(\eta,s_0,f_1,f_2) \leq 1$.

   Let us analyse the difference of $E_1$. From inequality $|\exp(-x)-\exp(-y)|\leq |x-y|$  we have
   \begin{equation}\label{demo:difE1}
       \begin{array}{l}
   \left| E_1(\eta,s_0,f_1,f_2)-E_1(\eta,s_0,g_1,g_2)\right| 
   \leq \bint_{s_0}^{\eta} 2av \left| \tfrac{N(f_1(v))}{L(f_1(v))}-\tfrac{N(g_1(v))}{L(g_1(v))}\right| dv\\ \\
   + D_1 \bint_{s_0}^{\eta}  \left| \tfrac{K(f_1(v))}{H(r_0,s_0,f_1,f_2)L(f_1(v)) v^{\nu}}-\tfrac{K(g_1(v))}{H(r_0,s_0,g_1,g_2)L(g_1(v)) v^{\nu}}\right| dv.
   \end{array}
   \end{equation}
On one hand,
\begin{equation}\label{demo:difE1-aux1}
    \begin{array}{l}
     \bint_{s_0}^{\eta} 2av \left| \tfrac{N(f_1(v))}{L(f_1(v))}-\tfrac{N(g_1(v))}{L(g_1(v))}\right| dv \\ \\
     \leq 2a  \bint_{s_0}^{\eta}  \left\lvert v \tfrac{N(f_1(v))L(g_1(v))-N(g_1(v))L(g_1(v))+N(f_1(v))L(g_1(v))-N(g_1(v))L(f_1(v)) }{L(f_1(v))L(g_1(v)) }\right\rvert   dv\\ \\
     \leq 2a\left[ \bint_{s_0}^{\eta}  \tfrac{v \wN_1 ||\vec{f}-\vec{g}||}{L_{1m} v^{\mu}}dv+\bint_{s_0}^{\eta}  \tfrac{ N_{1M} \wL_1 v^{1-\mu}  ||\vec{f}-\vec{g}||}{L_{1m}^2 v^{2\mu}}dv \right]\\ \\
     \leq 2a \left[ \tfrac{ \wN_1 }{L_{1m} (\mu-2)}\tfrac{1}{s_0^{\mu-2}}+\tfrac{ N_{1M} \wL_1}{L_{1m}^2 (3\mu-2)} \tfrac{1}{s_0^{3\mu-2}}  \right]  ||\vec{f}-\vec{g}||
    \end{array}
\end{equation}
   On the other hand, 
   \begin{equation}\label{demo:difE1-aux1}
    \begin{array}{l}
    D_1 \bint_{s_0}^{\eta}  \left| \tfrac{K(f_1(v))}{H(r_0,s_0,f_1,f_2)L(f_1(v)) v^{\nu}}-\tfrac{K(g_1(v))}{H(r_0,s_0,g_1,g_2)L(g_1(v)) v^{\nu}}\right| dv\\ \\
    \leq   D_1\Bigg( \bint_{s_0}^{\eta}  \left| \tfrac{K(f_1(v))}{H(r_0,s_0,f_1,f_2)L(f_1(v)) v^{\nu}}-\tfrac{K(g_1(v))}{H(r_0,s_0,f_1,f_2)L(f_1(v)) v^{\nu}}\right| dv \\ \\
    +\bint_{s_0}^{\eta} \left|\tfrac{K(g_1(v))}{H(r_0,s_0,f_1,f_2)L(f_1(v)) v^{\nu}}    -\tfrac{K(g_1(v))}{H(r_0,s_0,g_1,g_2)L(g_1(v)) v^{\nu}}\right| dv\Bigg)\\ \\
    \leq   D_1\Bigg( \bint_{s_0}^{\eta}  \tfrac{|K(f_1(v))-K(g_1(v))|}{H(r_0,s_0,f_1,f_2)L(f_1(v)) v^{\nu}} dv
    \\ \\
    +\bint_{s_0}^{\eta} |K(g_1(v))|\left|\tfrac{1}{H(r_0,s_0,f_1,f_2)L(f_1(v)) v^{\nu}}    -\tfrac{1}{H(r_0,s_0,g_1,g_2)L(g_1(v)) v^{\nu}}\right| dv\Bigg).\\ \\
    \end{array}
\end{equation}
From assumptions \eqref{cotaLf1}, \eqref{cotaKf1}, \eqref{LipschitzK1} and inequalities \eqref{cotaHinf}, \eqref{difH} we get that
\begin{equation}\label{demo:difE1-aux3}
 \bint_{s_0}^{\eta}  \tfrac{|K(f_1(v))-K(g_1(v))|}{H(r_0,s_0,f_1,f_2)L(f_1(v)) v^{\nu}} dv \leq  \tfrac{\wK_1}{H_{inf}(r_0,s_0) L_{1m} (2\mu+\nu-1)}\tfrac{1}{s_0^{2\mu+\nu-1}} ||\vec{f}-\vec{g}||
   \end{equation}
   and
\begin{equation}\label{demo:difE1-aux4}
\begin{array}{l}
  \bint_{s_0}^{\eta} |K(g_1(v))|\left|\tfrac{1}{H(r_0,s_0,f_1,f_2)L(f_1(v)) v^{\nu}}    -\tfrac{1}{H(r_0,s_0,g_1,g_2)L(g_1(v)) v^{\nu}}\right| dv \\ \\
  \leq  K_{1M} \Bigg( \bint_{s_0}^{\eta} \tfrac{|H(r_0,s_0,g_1,g_2)-H(r_0,s_0,f_1,f_2)|}{H(r_0,s_0,f_1,f_2)L(f_1(v))H(r_0,s_0,g_1,g_2) } \tfrac{dv}{v^{\mu+\nu}} \\ \\
  +\bint_{s_0}^{\eta} \tfrac{|L(g_1(v))-L(f_1(v)) |}{L(f_1(v))H(r_0,s_0,g_1,g_2)L(g_1(v)) } \tfrac{dv}{v^{\mu+\nu}}\Bigg)\\ \\
  \leq \tfrac{K_{1M}}{H_{inf}(r_0,s_0) L_{1m}} \Bigg( \tfrac{\widetilde{H}(r_0,s_0)}{H_{inf}(r_0,s_0) (2\mu+\nu-1)}\tfrac{1}{s_0^{2\mu+\nu-1}}+ \tfrac{\widetilde{L}_1}{ L_{1m} (3\mu+\nu-1)}\tfrac{1}{s_0^{3\mu+\nu-1}} \Bigg) ||\vec{f}-\vec{g}||.
  \end{array}
\end{equation}
Then inequalities \eqref{demo:difE1}-\eqref{demo:difE1-aux4} yields to
$$\left| E_1(\eta,s_0,f_1,f_2)-E_1(\eta,s_0,g_1,g_2)\right| \leq \widetilde{E}_1 (r_0,s_0) ||\vec{f}-\vec{g}||$$
with $\widetilde{E}_1$ given by \eqref{E1raya}.

From the definition of $\Phi_1$ given by \eqref{phi1} we have that
$$\begin{array}{ll}
|\Phi_1(\eta,s_0,f_1,f_2)-\Phi_1(\eta,s_0,g_1,g_2)|  \leq 
\bint_{s_0}^{\eta} \left| \tfrac{E_1(v,s_0,f_1,f_2)}{L(f_1(v))}-\tfrac{E_1(v,s_0,g_1,g_2)}{L(g_1(v))}\right| \tfrac{dv}{v^{\nu}} \\ \\
\leq \bint_{s_0}^{\eta}   \tfrac{|E_1(v,s_0,f_1,f_2)-E_1(v,s_0,g_1,g_2)|}{L(f_1(v))} \tfrac{dv}{v^{\nu}}
+  \bint_{s_0}^{\eta}  \tfrac{E_1(v,s_0,g_1,g_2) |L(f_1(v))-L(g_1(v))|}{L(f_1(v))L(g_1(v))}\tfrac{dv}{v^{\nu}} \\ \\
\leq \left(\tfrac{\wE_1(r_0,s_0) }{L_{1m}} \bint_{s_0}^{\eta} \tfrac{1}{v^{\mu+\nu}} dv + \tfrac{\wL_1}{L_{1m}^2} \bint_{s_0}^{\eta} \tfrac{1}{v^{2\mu+\nu}} dv\right) ||\vec{f}-\vec{g}|| \leq \widetilde{\Phi}_1(r_0,s_0)  ||\vec{f}-\vec{g}||
\end{array}
$$
where  $\widetilde{\Phi}_1(r_0,s_0)$ is given by \eqref{Phi1raya}.

Taking into account the definition of $H_1$ given by \eqref{defH1} we esasily obtain that
$$H_1(\eta,s_0,f_1,f_2) \geq K_{1m} \bint_{s_0}^{\eta} \tfrac{1}{v^{\mu+\nu}} dv\geq 
\tfrac{K_{1m}}{(\mu+\nu-1)} \left(\tfrac{1}{s_0^{\mu+\nu-1}}-\tfrac{1}{\eta^{\mu+\nu-1}}\right),$$

$$|H_1(\eta,s_0,f_1,f_2)| \leq \tfrac{K_{1M}}{E_{1inf}(r_0,s_0)} \bint_{s_0}^{\eta} \tfrac{1}{v^{\mu+\nu}} dv\leq  \tfrac{K_{1M}}{E_{1inf}(r_0,s_0)} \tfrac{1}{(\mu+\nu-1)} \tfrac{1}{s_0^{\mu+\nu-1}},$$
then \eqref{cotaH1inf} and  \eqref{cotaH1sup} hold.
In addition
$$\begin{array}{ll}
|H_1(\eta,s_0,f_1,f_2)-H_1(\eta,s_0,g_1,g_2)|  \leq 
\bint_{s_0}^{\eta} \left| \tfrac{K(f_1(v))}{E_1(v,s_0,f_1,f_2)}-\tfrac{K(g_1(v))}{E_1(v,s_0,g_1,g_2)}\right| \tfrac{dv}{v^{\nu}} \\ \\
\leq \bint_{s_0}^{\eta}  \tfrac{|K(f_1(v))-K(g_1(v))|}{E_1(v,s_0,f_1,f_2)}\tfrac{dv}{v^{\nu}} +\bint_{s_0}^{\eta}  \tfrac{K(g_1(v))|E_1(v,s_0,f_1,f_2)-E_1(v,s_0,g_1,g_2)| }{E_1(v,s_0,f_1,f_2)E_1(v,s_0,g_1,g_2)}\tfrac{dv}{v^{\nu}} \\ \\
\leq \tfrac{1}{E_{1inf}(r_0,s_0)} \left( \wK_1+ \tfrac{K_{1M} \wE_1 (r_0,s_0)}{E_{1inf}(r_0,s_0)} \right) \bint_{s_0}^{\eta} \tfrac{1}{v^{\mu+\nu}} dv\; \;||\vec{f}-\vec{g}|| \leq \widetilde{H}_1(r_0,s_0)||\vec{f}-\vec{g}||
\end{array}
$$
where $\widetilde{H}_1$ is given by \eqref{H1raya}.

From the definition of $G_1$ it follows that
$$\begin{array}{ll} G_1(\eta,s_0,f_1,f_2)|\geq \tfrac{E_{1inf}(r_0,s_0)}{L_{1M}} \bint _{s_0}^{\eta} \tfrac{H_{1inf}(v,s_0)}{v^{\mu+\nu}}dv\\ 
\geq \tfrac{K_{1m}E_{1inf}(r_0,s_0)}{L_{1M}(\mu+\nu-1) } \bint_{s_0}^{\eta} \tfrac{1}{v^{\mu+\nu}} \left(\tfrac{1}{s_0^{\mu+\nu-1}}-\tfrac{1}{v^{\mu+\nu-1}} \right) dv\geq G_{1inf}(\eta,r_0,s_0),\\ 
\end{array}$$
where $ G_{1inf}$ is given by \eqref{G1inf} and
$$\begin{array}{ll}|G_1(\eta,s_0,f_1,f_2)|\leq \bint\limits_{s_0}^{\eta} \left|\tfrac{E_1(v,s_0,f_1,f_2)H_1(v,s_0,f_1,f_2)}{L(f_1(v))} \right|\tfrac{dv}{v^{\nu}} \\ \\
\leq \tfrac{H_{1sup}(r_0,s_0)}{L_{1m}}  \bint_{s_0}^{\eta}  \tfrac{dv}{v^{\mu+\nu}} \leq G_{1sup}(r_0,s_0) ,
\end{array}$$
where $ G_{1sup}$ is given by \eqref{G1sup}.
Moreover
$$\begin{array}{ll}
     |G_1(\eta,s_0,f_1,f_2)-G_1(\eta,s_0,g_1,g_2)| \\ \\
  \leq  \bint_{s_0}^{\eta}  |H_1(v,s_0,f_1,f_2) |\left|\tfrac{E_1(v,s_0,g_1,g_2)}{L(g_1(v))}-\tfrac{E_1(v,s_0,f_1,f_2)}{L(f_1(v))} \right| \tfrac{dv}{v^{\nu}}\\ \\
  + \bint_{s_0}^{\eta}  \tfrac{E_1(v,s_0,g_1,g_2) |H_1(v,s_0,f_1,f_2)-H_1(v,s_0,g_1,g_2)|}{L(g_1(v))} \tfrac{dv}{v^{\nu}} \\ \\
  \leq \widetilde{G}_1(r_0,s_0) ||\vec{f}-\vec{g}||,
\end{array}$$
with $\widetilde{G}_1 $ defined by \eqref{G 1 raya}.
    
\end{proof}
\begin{lemma}\label{lema-cotaFase2} For every $\vec{f}=(f_1,f_2)$,  $\vec{g}=(g_1,g_2)\in\mathcal{K}$, the following inequalities hold: 
\begin{enumerate}

     \item 
    \begin{equation}\label{cotaE2inf}
        E_2(\eta,r_0,f_1,f_2)\geq E_{2 inf}(r_0,s_0),
    \end{equation}
    \begin{equation}\label{cotaE2sup}
        E_2(\eta,r_0,f_1,f_2)\leq 1,
    \end{equation}
    \begin{equation}\label{difE2}
        |E_2(\eta,r_0,f_1,f_2)-E_2(\eta,r_0,g_1,g_2)|\leq \wE_2(r_0,s_0) ||\vec{f}-\vec{g}||,
    \end{equation}
where
 \begin{equation}\label{E2inf}
 E_{2inf}(r_0,s_0):=\exp\left(-\left[  a  \tfrac{N_{2M} }{L_{2m} (\mu-1)} \tfrac{1}{ r_0^{2\mu-2}}+\tfrac{D_2 K_{2M} }{H_{inf}(r_0,s_0) L_{2m}(2\mu+\nu-1)} \tfrac{1}{r_0^{2\mu+\nu-1}}\right]\right),
 \end{equation}
 \begin{equation}\label{E2raya}
 \begin{array}{l}
 \widetilde{E}_2(r_0,s_0):= 2a \left[ \tfrac{ \wN_2 }{L_{2m} (\mu-2)}\tfrac{1}{r_0^{\mu-2}}+\tfrac{ N_{2M} \wL_2}{L_{2m}^2 (3\mu-2)} \tfrac{1}{r_0^{3\mu-2}}  \right]\\ \\
 +D_2\Big(  \tfrac{\wK_2}{H_{inf}(r_0,s_0) L_{2m} (2\mu+\nu-1)}\tfrac{1}{r_0^{2\mu+\nu-1}}\\ \\+\tfrac{K_{2M}}{H_{inf}(r_0,s_0) L_{2m}} \Big( \tfrac{\widetilde{H}(r_0)}{H_{inf}(r_0,s_0) (2\mu+\nu-1)}\tfrac{1}{r_0^{2\mu+\nu-1}}+ \tfrac{\widetilde{L}_2}{ L_{2m} (3\mu+\nu-1)}\tfrac{1}{r_0^{3\mu+\nu-1}} \Big)\Big).
 \end{array}
 \end{equation}

\item 
\begin{equation}\label{cotaPhi2inf}
         \Phi_2(\eta,r_0,f_1,f_2)\geq \Phi_{2 inf}(\eta,r_0,s_0),
    \end{equation}
    \begin{equation}\label{cotaPhi2sup}
      \Phi_2(\eta,r_0,f_1,f_2)\leq \Phi_{2sup}(r_0),
    \end{equation}
    \begin{equation}\label{Resta Phi 2}
|\Phi_2(\eta,r_0,f_1,f_2)-\Phi_2(\eta,r_0,g_1,g_2)|\leq \widetilde{\Phi}_2(r_0,s_0)||\vec{f}-\vec{g}||,
\end{equation}
  where  \begin{equation}\label{Phi2inf}
\Phi_{2 inf}(\eta,r_0,s_0):= \tfrac{E_{2inf}(r_0,s_0)}{L_{2M}} \tfrac{1}{(\mu+\nu-1)} \left(\tfrac{1}{r_0^{\mu+\nu-1}}-\tfrac{1}{\eta^{\mu+\nu-1}}\right),
  \end{equation}  
   \begin{equation}\label{Phi2sup}
\Phi_{2 sup}(r_0):= \tfrac{1}{L_{2m}} \tfrac{1}{(\mu+\nu-1)}\tfrac{1}{r_0^{\mu+\nu-1}},
  \end{equation}  
 \begin{equation}\label{Phi2raya}
 \widetilde{\Phi}_2(r_0,s_0):=\tfrac{\wE_2(r_0,s_0)}{L_{2m}} \tfrac{1}{(\mu+\nu-1)}\tfrac{1}{r_0^{\mu+\nu-1}}+ \tfrac{\wL_2}{L_{2m}^2}  \tfrac{1}{(2\mu+\nu-1)}\tfrac{1}{r_0^{2\mu+\nu-1}}.
 \end{equation}

\item
 \begin{equation}\label{cotaH2inf}
      H_2(\eta,r_0,f_1,f_2)\leq H_{2inf}(\eta,r_0,s_0),
    \end{equation}
  \begin{equation}\label{cotaH2sup}
      H_2(\eta,r_0,f_1,f_2)\leq H_{2sup}(r_0,s_0),
    \end{equation}
    \begin{equation}\label{Resta H2}
|H_2(\eta,r_0,f_1,f_2)-H_2(\eta,r_0,g_1,g_2)|\leq \widetilde{H}_2(r_0,s_0)||\vec{f}-\vec{g}||,
\end{equation}
 where  
 \begin{equation}\label{H2inf}
H_{2 inf}(\eta,r_0):=\tfrac{K_{2m}}{(\mu+\nu-1)} \left( \tfrac{1}{r_0^{\mu+\nu-1}}-\tfrac{1}{\eta^{\mu+\nu-1}}\right),
  \end{equation}  
    \begin{equation}\label{H2sup}
H_{2 sup}(r_0,s_0):= \tfrac{K_{2M}}{E_{2inf}(r_0,s_0)} \tfrac{1}{(\mu+\nu-1)} \tfrac{1}{r_0^{\mu+\nu-1}},
\end{equation}
 \begin{equation}\label{H2raya}
 \widetilde{H}_2(r_0,s_0):= \left( \wK_2+ \tfrac{K_{2M} \wE_2(r_0,s_0)}{E_{2inf}(r_0,s_0)}\right) \tfrac{1}{E_{2inf}(r_0,s_0) (\mu+\nu-1)} \tfrac{1}{r_0^{\mu+\nu-1}}.
 \end{equation}

\item 
\begin{equation}\label{cotaG2inf}
      G_2(\eta,r_0,f_1,f_2)\leq G_{2inf}(\eta,r_0,s_0),
    \end{equation}
 \begin{equation}\label{cotaG2sup}
      G_2(\eta,r_0,f_1,f_2)\leq G_{2sup}(r_0,s_0),
    \end{equation}
    \begin{equation}\label{Resta G2}
|G_2(\eta,r_0,f_1,f_2)-G_2(\eta,r_0,g_1,g_2)|\leq \widetilde{G}_2(r_0,s_0)||\vec{f}-\vec{g}||,
\end{equation}
 where 
 \begin{equation}\label{G2inf}
G_{2 inf}(\eta,r_0,s_0):= \tfrac{K_{2m} E_{2inf}(r_0,s_0) }{2 L_{2M} (\mu+\nu-1)^2} \left(\tfrac{1}{r_0^{\mu+\nu-1}}-\tfrac{1}{\eta^{\mu+\nu-1}} \right)^2,
\end{equation}
    \begin{equation}\label{G2sup}
G_{2 sup}(r_0,s_0):= \tfrac{H_{2sup}(r_0,s_0)}{L_{2m}} \tfrac{1}{(\mu+\nu-1)}\tfrac{1}{r_0^{\mu+\nu-1}},
\end{equation}
 \begin{equation}\label{G2raya}
 \widetilde{G}_2(r_0,s_0):= H_{2sup}(r_0,s_0) \widetilde{\Phi}_2(r_0,s_0)+ \tfrac{\widetilde{H}_2(r_0,s_0)}{L_{2m}}  \tfrac{1}{(\mu+\nu-1)}\tfrac{1}{r_0^{\mu+\nu-1}}.
 \end{equation}

\end{enumerate}
\end{lemma}

\begin{proof}
    The proof follows  analogously to the previous lemma.
\end{proof}

\begin{lemma}
For every $\vec{f}=(f_1,f_2),  \vec{g}=(g_1,g_2)\in\mathcal{K}$ it follows that
$$||V_1(\vec{f})-V_1(\vec{g}) ||_{C[s_o,r_0]}\leq  \varepsilon_1(r_0,s_0)||\vec{f} -\vec{g}||$$
where
\begin{equation}\label{epsilon1}
    \varepsilon_1(r_0,s_0)=2s_0^{\nu}Q\exp(-s_0^{2})\widetilde{\Phi}_1(r_0,s_0)+2D_1^*\bigg(\tfrac{G_{1sup}(r_0,s_0)2H_{sup}(r_0,s_0)\wH(r_0,s_0) }{H^{4}_{inf}(r_0,s_0)}+\tfrac{\widetilde{G}_1(r_0,s_0)}{H^{2}_{inf}(r_0,s_0)}\bigg)
    \end{equation}
    \end{lemma}
\begin{proof}
Taking into account that
\begin{equation}\label{Diferencia cociente G1 H}
\begin{array}{l}
\bigg|\frac{G_1(\eta,s_0,f_1,f_2)}{H^{2}(r_0,s_0,f_1,f_2)}-\frac{G_1(\eta,s_0,g_1,g_2)}{H^{2}(r_0,s_0,g_1,g_2)}\bigg|\leq

\frac{|G_1(\eta,s_0,f_1,f_2)| |H^{2}(r_0,s_0,f_1,f_2)-H^{2}(r_0,s_0,g_1,g_2)|}{H^{2}(r_0,s_0,f_1,f_2)H^{2}(g_1,g_2)}\\[0.40cm]+\frac{|G_1(\eta,s_0,f_1,f_2)-G_1(\eta,s_0,g_1,g_2)| }{H^{2}(r_0,s_0,g_1,g_2)}\leq \bigg[\frac{G_{1sup}(r_0,s_0)2H_{sup}(r_0,s_0)\wH(r_0,s_0) }{H^{4}_{inf}(r_0,s_0)}+\frac{\widetilde{G}_1(r_0,s_0)}{H^{2}_{inf}(r_0,s_0}\bigg]||\vec{f}-\vec{g}||\\

\end{array}
\end{equation}
then, for each $\eta\in [s_0,r_0]$ it follows that
\begin{equation}\label{Resta V1}
    \begin{array}{l}
    |V_1(\vec{f})(\eta)-V_1(\vec{g})(\eta)|\leq \\[0.25cm]
   
    \leq s_0^{\nu}Q\exp(-s_0^{2}) \big[|\Phi_1(r_0,s_0,f_1,f_2)-\Phi_1(r_0,s_0,g_1,g_2)|\\ \\
    + |\Phi_1(\eta,s_0,f_1,f_2)-\Phi_1(\eta,s_0,g_1,g_2)|\big]
  \\ \\
       + \big|\frac{D_1^*G_1(r_0,s_0,f_1,f_2)}{H^{2}(r_0,s_0,f_1,f_2)}-\frac{D_1^*G_1(r_0,s_0,g_1,g_2)}{H^{2}(r_0,s_0,g_1,g_2)}\bigg|+ \bigg|\frac{D_1^*G_1(\eta,s_0,f_1,f_2)}{H^{2}(r_0,s_0,f_1,f_2)}-\frac{D_1^*G_1(\eta,s_0,g_1,g_2)}{H^{2}(r_0,s_0,g_1,g_2)}\big|, \\ \\
       \leq \Big[2s_0^{\nu}Q\exp(-s_0^{2})\widetilde{\Phi}_1(r_0,s_0)+2D_1^*\big(\frac{G_{1sup}(r_0,s_0)2H_{sup}(r_0,s_0)\wH(r_0,s_0) }{H^{4}_{inf}(r_0,s_0)}+\frac{\widetilde{G}_1(r_0,s_0)}{H^{2}_{inf}(r_0,s_0)}\big)\Big]||\vec{f}-\vec{g}||\\ \\
       =\varepsilon_1(r_0,s_0)||\vec{f}-\vec{g}||.
\end{array}
\end{equation}

\end{proof}

\begin{lemma}
For every $\vec{f}=(f_1,f_2),  \vec{g}=(g_1,g_2)\in\mathcal{K}$ it follows that
$$||V_2(\vec{f})-V_2(\vec{g}) ||_{C_b[r_0,+\infty)}\leq  \varepsilon_2(r_0,s_0)||\vec{f} -\vec{g}||$$
where
\begin{equation}\label{epsilon2}
\varepsilon_2(r_0,s_0)=\varepsilon_{21}(r_0,s_0)+\varepsilon_{22}(r_0,s_0)+\varepsilon_{23}(r_0,s_0)
\end{equation}
with
$$\begin{array}{ll}
\varepsilon_{21}(r_0,s_0)=\frac{2 \widetilde{\Phi}_2(r_0,s_0)}{\Phi_{2inf}(+\infty,r_0,s_0)} , \\ \\
\varepsilon_{22}(r_0,s_0)= \frac{\widetilde{G}_2(r_0,s_0)}{H^2_{inf}(r_0,s_0)}+ \frac{2 G_{2sup}(r_0,s_0) H_{sup}(r_0,s_0) \wH(r_0,s_0)}{H^4_{inf}(r_0,s_0)} ,\\ \\
\varepsilon_{23}(r_0,s_0)=\frac{\Phi_{2sup}(r_0,s_0)}{\Phi_{2inf}(+\infty,r_0,s_0)} \varepsilon_{22}(r_0,s_0)+ \frac{G_{2sup}(r_0,s_0)}{H^2_{inf}(r_0,s_0)}\varepsilon_{21}(r_0,s_0).
\end{array}$$
\end{lemma}

\begin{proof}
    
On one hand we have that
\begin{equation}\label{Dif Cociente Phi 2}
    \begin{array}{ll}
         \bigg| \frac{\Phi_2(\eta,r_0,f_1,f_2)}{\Phi_2(+\infty,r_0,f_1,f_2)}-\frac{\Phi_2(\eta,r_0,g_1,g_2)}{\Phi_2(+\infty,r_0,g_1,g_2)}\bigg|\leq  \frac{| \Phi_2(\eta,r_0,f_1,f_2)-\Phi_2(\eta,r_0,g_1,g_2) |}{\Phi_2(+\infty,r_0,f_1,f_2)}\\[0.40cm]
         + \frac{\Phi_2(\eta,r_0,g_1,g_2)}{\Phi_2(+\infty,r_0,g_1,g_2)}\frac{| \Phi_2(+\infty,r_0,f_1,f_2)-\Phi_2(+\infty,r_0,g_1,g_2) |}{\Phi_2(+\infty,r_0,f_1,f_2)} \\[0.40cm]
         \leq \frac{| \Phi_2(\eta,r_0,f_1,f_2)-\Phi_2(\eta,r_0,g_1,g_2) |}{\Phi_2(+\infty,r_0,f_1,f_2)}+\frac{| \Phi_2(+\infty,r_0,f_1,f_2)-\Phi_2(+\infty,r_0,g_1,g_2) |}{\Phi_2(+\infty,r_0,f_1,f_2)} \\[0.40cm]
         \leq \frac{2 \widetilde{\Phi}_2(r_0,s_0)}{\Phi_{2inf}(+\infty,r_0,s_0)} ||\vec{f}-\vec{g}||=\varepsilon_{21}(r_0,s_0)||\vec{f}-\vec{g}||.
    \end{array}
\end{equation}
On the other hand we obtain that
\begin{equation}\label{Dif Cociente G2/Hcuad}
    \begin{array}{ll}
         \bigg| \frac{G_2(\eta,r_0,f_1,f_2)}{H^2(s_0,r_0,f_1,f_2)}-\frac{G_2(\eta,r_0,g_1,g_2)}{H^2(s_0,r_0,g_1,g_2)}\bigg| \\ \\
         \leq  \frac{|G_2(\eta,r_0,f_1,f_2)-G_2(\eta,r_0,g_1,g_2)|}{H^2(s_0,r_0,f_1,f_2)}+ \frac{|G_2(\eta,r_0,g_1,g_2)||H^2(s_0,r_0,g_1,g_2)-H^2(s_0,r_0,f_1,f_2)|}{H^2(s_0,r_0,f_1,f_2) H^2(s_0,r_0,g_1,g_2)} \\ \\
         \leq \left(  \frac{\widetilde{G}_2(r_0,s_0)}{H^2_{inf}(r_0,s_0)}+ \frac{2 G_{2sup}(r_0,s_0) H_{sup}(r_0,s_0) \wH(r_0,s_0)}{H^4_{inf}(r_0,s_0)}\right) ||\vec{f}-\vec{g}||\\ \\
         =\varepsilon_{22}(r_0,s_0)||\vec{f}-\vec{g}||.
    \end{array}
\end{equation}
In addition 
\begin{equation}\label{Dif Cociente Phi 2 con G2/Hcuad}
    \begin{array}{ll}
         \bigg| \frac{G_2(+\infty,r_0,f_1,f_2)}{H^2(s_0,r_0,f_1,f_2)}\frac{\Phi_2(\eta,r_0,f_1,f_2)}{\Phi_2(+\infty,r_0,f_1,f_2)}-\frac{G_2(+\infty,r_0,g_1,g_2)}{H^2(s_0,r_0,g_1,g_2)} \frac{\Phi_2(\eta,r_0,g_1,g_2)}{\Phi_2(+\infty,r_0,g_1,g_2)}\bigg| \\ \\
         \leq \frac{\Phi_2(\eta,r_0,f_1,f_2)}{\Phi_2(+\infty,r_0,f_1,f_2)} \bigg| \frac{G_2(+\infty,r_0,f_1,f_2)}{H^2(s_0,r_0,f_1,f_2)}-\frac{G_2(+\infty,r_0,g_1,g_2)}{H^2(s_0,r_0,g_1,g_2)}\bigg| \\ \\
         + \frac{G_2(+\infty,r_0,g_1,g_2)}{H^2(s_0,r_0,g_1,g_2)} \bigg| \frac{\Phi_2(\eta,r_0,f_1,f_2)}{\Phi_2(+\infty,r_0,f_1,f_2)}-\frac{\Phi_2(\eta,r_0,g_1,g_2)}{\Phi_2(+\infty,r_0,g_1,g_2)}\bigg| \\ \\
         \leq \left(\frac{\Phi_{2sup}(r_0,s_0)}{\Phi_{2inf}(+\infty,r_0,s_0)} \varepsilon_{22}(r_0,s_0)+ \frac{G_{2sup}(r_0,s_0)}{H^2_{inf}(r_0,s_0)}\varepsilon_{21}(r_0,s_0)\right)||\vec{f}-\vec{g}||\\ \\
         = \varepsilon_{23}(r_0,s_0)||\vec{f}-\vec{g}||.
    \end{array}
\end{equation}
Taking into account the previous inequalities, for each $\eta\geq r_0$,  it follows   that
\begin{equation}\label{Resta V2}
    \begin{array}{l}
    |V_2(\vec{f})(\eta)-V_2(\vec{g})(\eta)| \\[0.25cm]
  \leq  \bigg|
    \frac{D_2^*G_2(+\infty,r_0,f_1,f_2)}{H^{2}(r_0,s_0,f_1,f_2)}\frac{\Phi_2(\eta,r_0,f_1,f_2)}{\Phi_2(+\infty,r_0,f_1,f_2)}-\frac{D_2^*G_2(+\infty,r_0,g_1,g_2)}{H^{2}(r_0,s_0,g_1,g_2)}\frac{\Phi_2(\eta,r_0,g_1,g_2)}{\Phi_2(+\infty,r_0,g_1,g_2)}\bigg| \\[0.25cm] 
    +\bigg| \frac{\Phi_2(\eta,r_0,f_1,f_2)}{\Phi_2(+\infty,r_0,f_1,f_2)}-\frac{\Phi_2(\eta,r_0,g_1,g_2)}{\Phi_2(+\infty,r_0,g_1,g_2)}\bigg|
    + \bigg|\frac{D_2^*G_2(\eta,r_0,f_1,f_2)}{H^{2}(r_0,s_0,f_1,f_2)}-\frac{D_2^*G_2(\eta,r_0,g_1,g_2)}{H^{2}(r_0,s_0,g_1,g_2)}\bigg|\\[0.25cm] 
    \leq \varepsilon_2(r_0,s_0)||\vec{f}-\vec{g}||.
\end{array}
\end{equation}

\end{proof}

\begin{theorem}
For every $\vec{f}=(f_1,f_2)$,  $\vec{g}=(g_1,g_2)\in\mathcal{K}$ it follows that
$$|| \Psi(\vec{f})-\Psi(\vec{g}) ||\leq  \varepsilon(r_0,s_0)||\vec{f} -\vec{g}||$$
with
\begin{equation}\label{epsilon}
    \varepsilon(r_0,s_0)=\max \left\lbrace \varepsilon_1(r_0,s_0),\varepsilon_2(r_0,s_0) \right\rbrace
\end{equation}
where $\varepsilon_1(r_0,s_0)$ and $\varepsilon_2(r_0,s_0)$ are given by \eqref{epsilon1} and \eqref{epsilon2}, respectively.
\end{theorem}

\begin{proof}
From the previous lemmas we have that
$$\begin{array}{ll}|| \Psi(\vec{f})-\Psi(\vec{g}) ||=\max\left\lbrace ||V_1(\vec{f})-V_1(\vec{g}) ||_{C[s_0,r_0]},||V_2(\vec{f})-V_2(\vec{g}) ||_{C_b[r_0,+\infty)} \right\rbrace \\ \\
=\max\left\lbrace \varepsilon_1(r_0,s_0) ||\vec{f}-\vec{g} ||, \varepsilon_2(r_0,s_0) ||\vec{f}-\vec{g} || \right\rbrace = \varepsilon(r_0,s_0) ||\vec{f}-\vec{g} ||.
\end{array}$$

\end{proof}

Now we will look for conditions that guarantee that $\Psi$ is a contraction mapping.

For each $s_0>0$ fixed, we define the functions $\varepsilon_{1,s_0}(r_0)=\varepsilon_{1}(r_0,s_0)$ and $\varepsilon_{2,s_0}(r_0)=\varepsilon_{2}(r_0,s_0)$, for all $r_0>s_0$ where $\varepsilon_{1} $, $\varepsilon_{2} $ are given by \eqref{epsilon1} and \eqref{epsilon2}, respectively. 
The following results hold:
\begin{lemma} \label{lema-epsilon1}$ $
   \begin{enumerate}
       \item[a)] The function $\varepsilon_{1,s_0}$ is a decreasing function that satisfies $\varepsilon_{1,s_0}(s_0)=+\infty$ and $\varepsilon_{1,s_0}(+\infty)=j_1(s_0)$ where
       \begin{equation}\label{j1}
       \begin{array}{ll}
           j_1(s_0)=2s_0^{\nu}Q\exp(-s_0^{2})\widetilde{\Phi}_1(+\infty,s_0)\\ \\
           +2D_1^*\bigg(\tfrac{G_{1sup}(+\infty,s_0)2H_{sup}(+\infty,s_0)\wH(+\infty,s_0) }{H^{4}_{inf}(+\infty,s_0)}+\tfrac{\widetilde{G}_1(+\infty,s_0)}{H^{2}_{inf}+\infty,s_0)}\bigg)
           \end{array}
       \end{equation}
       \item[b)]   If 
    \begin{equation}\label{hip-epsilon1}
        \tfrac{2D_1^* \wK_1}{L_{1m} K_{1m}^2} \left( \tfrac{2 K_{1M}}{K_{1m}^2}+1 \right)<1
    \end{equation}
    then there exists a unique $s_1>0$ such that $j_1(s_0)<1$ for all $s_0>s_1$.
    Moreover, for each $s_0>s_1$ there exists $r_1=r_1(s_0)>s_0$ such that $\varepsilon_{1,s_0}(r_1)=1$  and $\varepsilon_{1,s_0}(r_0)<1$ for all $r_0>r_1$. 
   \end{enumerate}
\end{lemma}

\begin{proof}  $   $
    \begin{enumerate}

        \item[a)] According to the definition of $\varepsilon_1$ given by \eqref{epsilon1}, the proof follows straightforwardly from Lemmas \ref{lema-cotaH} and \ref{lema-cotaFase1}.
        
        \item[b)] From the definition of $j_1$ given by \eqref{j1}, we have that it is a decreasing function that satisfies $j_1(0)=+\infty$ and $j_1(+\infty)=\tfrac{2D_1^* \wK_1}{L_{1m} K_{1m}^2} \left( \tfrac{2 K_{1M}}{K_{1m}^2}+1 \right)$. Then, assuming \eqref{hip-epsilon1} it follows that there exists a unique $s_1>s_0$ such that $j_1(s_1)=1$ and  $j_1(s_0)<1$ for all $s_0>s_1$. Moreover, from item a), for each $s_0>s_1$ there exists $r_1=r_1(s_0)>s_0$ such that $\varepsilon_{1,s_0}(r_1)=1$  and $\varepsilon_{1,s_0}(r_0)<1$ for all $r_0>r_1$. 
    \end{enumerate}
\end{proof}

\begin{lemma}\label{lema-epsilon2} $ $
   \begin{enumerate}
       \item[a)] The function $\varepsilon_{2,s_0}$ is a decreasing function that satisfies $\varepsilon_{2,s_0}(s_0)=+\infty$ and $\varepsilon_{2,s_0}(+\infty)=0$.

    \item[b)] For each $s_0>0$ there exists $r_2=r_2(s_0)>s_0$ such that $\varepsilon_{2,s_0}(r_2)=1$  and $\varepsilon_{2,s_0}(r_0)<1$ for all $r_0>r_2$. 
   \end{enumerate}
\end{lemma}

\begin{proof} $ $
    \begin{enumerate}
        \item[a)] It follows from Lemmas \ref{lema-cotaH} and \ref{lema-cotaFase2}  taking into account that $\varepsilon_2$ is defined  by \eqref{epsilon2}.

        \item[b)] It is clear from item a)
    \end{enumerate}
\end{proof}

\begin{theorem}
  If inequality \eqref{hip-epsilon1} holds, then for each $(r_0,s_0)\in \Sigma$ with
    \begin{equation}\label{sigma}
        \Sigma=\lbrace (r_0,s_0) : \; s_0>s_1,\;r_0>\overline{r}_0(s_0) \rbrace
    \end{equation}
     we have that $\varepsilon(r_0,s_0)<1$ where $\varepsilon$ is given by \eqref{epsilon} and  
     \begin{equation}\label{r0raya}
     \overline{r}_0(s_0)=\max\lbrace r_1(s_0),r_2(s_0)\rbrace
     \end{equation}
    with
     $s_1,r_1$ and $r_2$  defined in Lemmas \ref{lema-epsilon1} and \ref{lema-epsilon2}, respectively. 
\end{theorem}

\begin{proof}
    The proof follows immediately by Lemmas \ref{lema-epsilon1} and \ref{lema-epsilon2}.
\end{proof}

\begin{corollary}\label{Coro: Psi contractivo}
    Under the assumption \eqref{hip-epsilon1}, for each $(r_0,s_0)\in \Sigma$, the operator $\Psi$ defined by \eqref{Psi} is a contraction mapping.
\end{corollary}

 \begin{theorem} \label{teo punto fijo}
       Under the assumption \eqref{hip-epsilon1}, for each $(r_0,s_0)\in \Sigma$, there exists a unique fixed point $(f_1^*,f_2^*)\in\mathcal{K}$ to the operator $\Psi$.
 \end{theorem}   
\begin{proof}
First, notice that $\mathcal{K}$ is a closed subset of the Banach space $\mathcal{C}$ given by \eqref{mathcal C}. 
In addition, it is easy to see  that $\Psi(\vec{f})\in \mathcal{K}$ given that    $V_1(\vec{f})\in C[s_0,r_0]$, $V_2(\vec{f})\in C_b[r_0,+\infty)$, $V_2(\vec{f})(r_0)=0$ and $V_2(\vec{f})(+\infty)=0$.
Finally, according to Corollary \ref{Coro: Psi contractivo}, under the assumption \eqref{hip-epsilon1}, for each $(r_0,s_0)\in \Sigma$ it results that $\Psi$ is a contraction mapping.
As a consequence, applying the fixed point Banach theorem, there exists a unique  fixed point $(f_1^*,f_2^*)\in\mathcal{K}$ to the operator $\Psi$ for each  $(r_0,s_0)\in \Sigma$.
    
\end{proof}

\begin{corollary}
   If \eqref{hip-epsilon1} holds, for each  $(r_0,s_0)\in \Sigma$,  there exists a unique solution $(f_1^*,f_2^*)$ to the system of equations \eqref{eq41}-\eqref{eq42}.
\end{corollary}

It remains to prove the existence of solution $(r_0,s_0)\in\Sigma$ to the system of equations given by \eqref{Ec1-s0,r0} and  \eqref{Ec2-s0,r0} where $f_1=f_1^*$ and $f_2=f_2^*$ are the unique solutions to the equations \eqref{eq41}-\eqref{eq42}. For that purpose we will need some preliminary results.

Let us notice that equation \eqref{Ec1-s0,r0} can be rewritten as
\begin{equation}\label{ec reemplaza 43}
  X(r_0,s_0)=Y(r_0,s_0)
\end{equation}
where 
\begin{equation}
      X(r_0,s_0)= Z(r_0,s_0)-B,\quad   Z(r_0,s_0)=\tfrac{D_1^* G_1(r_0,s_0,f_1^*,f_2^*)}{H^2(r_0,s_0,f_1^*,f_2^*)}
\end{equation}
and
\begin{equation}\label{Y}\begin{array}{ll}
   Y(r_0,s_0)=-Q s_0^{\nu} \exp(-s_0^2) \Phi_1(r_0,s_0,f_1^*,f_2^*).
    \end{array}
\end{equation}

\begin{lemma}\label{lem prop ec 43}
The following properties hold:
 \begin{itemize}
     \item[a)] $Y(r_0,s_0)<0$ for each $(r_0,s_0)\in\Sigma$.

     \item[b)] $Z(r_0,s_0)>Z_{inf}(r_0,s_0)$ for each $(r_0,s_0)\in\Sigma$,
where
$$Z_{inf}(r_0,s_0)=\tfrac{D_1^* E_{1inf}(r_0,s_0) K_{1m}}{2L_{1M}} \left(\tfrac{r_0^{\mu+\nu-1}-s_0^{\mu+\nu-1}}{K_{1M}r_0^{\mu+\nu-1}+K_{2M}s_0^{\mu+\nu-1}}\right)^2.$$

\item[c)]  For a fixed $s_0>s_1$, if we assume
\begin{equation} \label{X <Y en ro raya}
    X(\overline{r}_0(s_0),s_0)<Y(\overline{r}_0(s_0),s_0)
\end{equation}
then $Z_{inf}(\cdot,s_0)$ is an increasing function that satisfies
 \begin{equation}
 Z_{inf}(\overline{r_0}(s_0),s_0)<B,\qquad Z_{inf}(+\infty,s_0)=j_2(s_0),
 \end{equation}   
 where
 \begin{equation}
   j_2(s_0)=  \tfrac{D_1^* E_{1inf}(+\infty,s_0) K_{1m}}{L_{1M}   K_{1M}^2},
 \end{equation}
 is an increasing function that satisfies $j_2(+\infty)=\tfrac{D_1^* K_{1m}}{2L_{1M}   K_{1M}^2}$.

 \item[e)]  If we assume
 \begin{equation}\label{hipZinfty}
     \tfrac{D_1^* K_{1m}}{2L_{1M}   K_{1M}^2}>B
 \end{equation}
 then there exists a unique $s_2=\min\lbrace s_0\geq s_1: j_2(s_0)\geq B\rbrace $. Moreover, for each $s_0>s_2$, we have that $j_2(s_0)>B$.

 \item[f)] If \eqref{X <Y en ro raya} and \eqref{hipZinfty} holds for each $s_0>s_2$ there exists a unique $r_{B}(s_0)>s_0$ such that $Z_{inf}(r_0,s_0)>B$ for all $r_0>r_{B}(s_0)$.

 \end{itemize}

\end{lemma}

\begin{proof}
    \item[a)] It is clear from the definition of the function $Y$ given by \eqref{Y}.

\item[b)] It follows from the inequalities obtained in Lemmas \ref{lema-cotaH} and \ref{lema-cotaFase1}.

\item[c)] From the definition of $Z_{inf}$ it it easy to see that $Z_{inf}(\cdot, s_0)$  is an increasing function for each fixed $s_0>s_1$. In addition, assumption \eqref{X <Y en ro raya} and item a) leads to 
$$Z_{inf}(\overline{r}_0(s_0),s_0)-B<Z(\overline{r}_0(s_0),s_0)-B<Y(\overline{r}_0(s_0),s_0)<0,$$
Then it follows that $Z_{inf}(\overline{r}_0(s_0),s_0)<B$.
Finally, taking a limit gives that $Z_{inf}(+\infty,s_0)=j_2(s_0)$ for each $s_0>s_1$.

\item[e)] First, notice that hypothesis \eqref{hipZinfty} can be rewritten as $j_2(+\infty)>B$. From the fact that $j_2$ is an increasing function, we can conclude that there exists a unique $s_2=\min\left\lbrace s_0\geq s_1 : j_2(s_0)\geq B\right\rbrace$. Notice that  $s_2=s_1$ in case that $j_2(s_1)>B$. As a consequence for each $s_0>s_2$, we get that  $j_2(s_0)>B$.

\item[f)] For each fixed $s_0>s_2$,  we have that $Z_{inf}(\overline{r}_0(s_0),s_0)<B$ from item c) and 
$Z_{inf}(+\infty,s_0)>B$ from item e). 
Then there exists a unique $r_B=r_B(s_0)>\overline{r}_0(s_0)$ such that $Z_{inf}(r_B(s_0),s_0)=B$ and $Z_{inf}(r_0,s_0)>B$ for all $r_0>r_B(s_0)$.
\end{proof}

\begin{lemma}
For each $s_0>s_2$, if we assume \eqref{X <Y en ro raya} and \eqref{hipZinfty} there exists at least one solution $r_0^*=r_0^*(s_0,f_1^*,f_2^*)\in \left( \overline{r}_0(s_0),r_B(s_0)\right)$ to the equation \eqref{Ec1-s0,r0}.
\end{lemma}

\begin{proof}

 For each $s_0>s_2$, taking into account assumption \eqref{X <Y en ro raya} and the fact that
 from item f) of Lemma \ref{lem prop ec 43} the following inequality holds
    $$X(r_B(s_0),s_0)\geq Z_{inf}(r_B(s_0),s_0)-B=0> Y(r_B(s_0),s_0)$$
    we obtain that  there exists at least one solution $r_0^*\in \left( \overline{r}_0(s_0),r_B(s_0)\right)$ to the equation \eqref{Ec1-s0,r0}.
\end{proof}

Now we will analyze  the equation \eqref{Ec2-s0,r0}. If we replace $r_0$ by $r_0^*(s_0)$ and $(f_1,f_2)$ by $(f_1^*,f_2^*)$,  it results equivalent to
\begin{equation}\label{ec W=M}
    W(r_0^*(s_0),s_0)=M
\end{equation}
where
\begin{equation}\begin{array}{l}
 W(r_0^*(s_0),s_0)=\tfrac{E_1(r_0^*(s_0),s_0,f_1^*,f_2^*)}{r_0^{\nu+1}}\Bigg[Q\exp(-s_0^{2})s_0^{\nu}+\frac{ D_1^*}{H^{2}(r_0^*(s_0),s_0,f_1^*,f_2^*)}H_1(r_0^*(s_0),s_0,f_1^*,f_2^*)\Bigg]  \\ \\
 -\frac{1}{r_0^*(s_0)^{\nu+1} \Phi_2(+\infty,r_0^*(s_0),f_1^*,f_2^*)}\Bigg[1-\frac{ D_2^*}{H^{2}(r_0^*(s_0),s_0,f_1^*,f_2^*)}G_2(+\infty,r_0^*(s_0),f_1^*,f_2^*)\Bigg].
 \end{array}
\end{equation}

\begin{lemma}
If any of the following system of inequalities hold
\begin{equation}\label{hip-ec 44}
    \left\lbrace\begin{array}{ll}
         W_{inf}(s_2)>M\\ \\
         W_{sup}(+\infty)<M
    \end{array}\right.
    \qquad \text{or}\qquad 
     \left\lbrace\begin{array}{ll}
         W_{sup}(s_2)<M\\ \\
         W_{inf}(+\infty)>M
    \end{array}\right.
\end{equation}
then there exists at least one solution $\widehat{s_0}>s_2$ to the equation  \eqref{ec W=M}, 
where
    \begin{equation}\begin{array}{l}
 W_{inf}(s_0)=\tfrac{E_{1inf}(r_0^*(s_0),s_0)}{r_B^{\nu+1}(s_0)}\Bigg[Q\exp(-s_0^{2})s_0^{\nu}+\frac{ D_1^*}{H_{sup}^{2}(r_0^*(s_0),s_0)} H_{1inf}(r_0^*(s_0),s_0)\Bigg]  \\ \\
 -\frac{1}{\overline{r_0}^{\nu+1}(s_0) \Phi_{2inf}(+\infty,r_0^*(s_0),s_0)}+\frac{1}{r_B^{\nu+1}(s_0) \Phi_{2sup}(r_0^*(s_0))}\frac{ D_2^*}{H_{sup}^{2}(r_0^*(s_0),s_0)}G_{2inf}(+\infty,r_0^*(s_0),s_0).
 \end{array}
\end{equation}
\begin{equation}\begin{array}{l}
 W_{sup}(s_0)=\tfrac{1}{\overline{r_0}^{\nu+1}(s_0)}\Bigg[Q\exp(-s_0^{2})s_0^{\nu}+\frac{ D_1^*}{H_{inf}^{2}(r_0^*(s_0),s_0)} H_{1sup}(r_0^*(s_0),s_0)  \\ \\
 +\frac{1}{\Phi_{2inf}(+\infty,r_0^*(s_0),s_0)}\frac{ D_2^*}{H_{inf}^{2}(r_0^*(s_0),s_0)}G_{2sup}(r_0^*(s_0),s_0).\Bigg].
 \end{array}
\end{equation}

\end{lemma}

The prior analysis allows to establish the following existence theorem.

\begin{theorem}
    Assuming (A1)-(A5), \eqref{hip-epsilon1}, \eqref{X <Y en ro raya}, \eqref{hipZinfty}   and \eqref{hip-ec 44} there exists at least one solution $(\widehat{s_0},r_0^*(\widehat{s_0}),f_1^*,f_2^*)$ where $(f_1^*,f_2^*)$ is the unique fixed point of the operator $\Psi$ corresponding to $(\widehat{s_0},r_0^*(\widehat{s_0}))\in \Sigma$.
\end{theorem}

\begin{corollary}
   If (A1)-(A5), \eqref{hip-epsilon1}, \eqref{X <Y en ro raya}, \eqref{hipZinfty}   and \eqref{hip-ec 44}  hold there exists at least one solution to the problem \eqref{eq4}-\eqref{eq16} where
    $$\left\lbrace\begin{array}{lll}
    T_1(z,t)=T_m f_1^*\left(\tfrac{z}{2a \sqrt{t}} \right)+T_m,\qquad  &s(t)\leq z\leq r(t),\; t>0 \\ \\
    T_2(z,t)=T_m f_2^*\left(\tfrac{z}{2a \sqrt{t}} \right)+T_m,\qquad &z\geq r(t),\; t>0 \\ \\
 \varphi_1(z,t)=\dfrac{U_c}{2}\cdot\dfrac{F_1\left(\tfrac{z}{2a \sqrt{t}} ,\widehat{s_0},f_1^*\right)}{H(r_0^*(\widehat{s_0}),\widehat{s_0},f_1^*,f_2^*)}\qquad &s(t)\leq z\leq r(t),\; t>0 \\ \\
     \varphi_2(z,t)=\dfrac{U_c}{2}\cdot\dfrac{F_1(r_0^*(\widehat{s_0}),\widehat{s_0},f_1^*)+F_2\left(\tfrac{z}{2a \sqrt{t}} ,r_0^*(\widehat{s_0}),f_2^*\right)}{H(r_0^*(\widehat{s_0}),\widehat{s_0},f_1^*,f_2^*)},\qquad &z\geq r(t),\; t>0 \\ \\
    \end{array}\right.
    $$
with $s(t)=2a\widehat{s_0} \sqrt{t}$ and $r(t)=2ar_0^*(\widehat{s_0}) \sqrt{t}$.
\end{corollary}

The contribution of the section on existence and uniqueness in our paper lies in the development of a comprehensive mathematical model that captures the intricate behavior of electromagnetic fields and heat transfer in closed electrical contacts. We address the complex phenomenon of instantaneous micro-asperity explosion, incorporating vaporization, liquid, and solid zones, while accounting for non-linear thermal coefficients and temperature-dependent electrical conductivity.

Our proposed solutions, based on similarity transformations, facilitate the reduction of the problem to ordinary differential equations, streamlining the analysis process. We rigorously establish the validity of this approach through discussions and proofs supported by fixed point theory in Banach space. This section not only offers theoretical insights into the existence and uniqueness of solutions but also provides practical methodologies for tackling the intricate dynamics of electrical contacts under extreme conditions.
\section{Conclusions}
In conclusion, this study has presented novel mathematical models to characterize the intricate dynamics of electromagnetic fields and heat transfer within closed electrical contacts, particularly focusing on the instantaneous explosion of micro-asperities. These explosions give rise to complex phenomena involving vaporization zones, as well as liquid and solid phases, where temperatures are governed by a generalized heat equation incorporating Thomson effect considerations.

One of the key strengths of our models lies in their ability to capture the nonlinear behavior exhibited by thermal coefficients and electrical conductivity, both of which are temperature-dependent. By employing similarity transformations, we have effectively reduced the problem to a set of ordinary differential equations, thereby facilitating tractable analysis and solution.

The validity and utility of our approach have been rigorously demonstrated through discussions and proofs grounded in fixed point theory within the framework of Banach spaces. This theoretical underpinning not only enhances our confidence in the proposed solutions but also provides a solid foundation for future research endeavors in related domains.

Furthermore, the insights gained from this study hold significant implications for various practical applications involving electrical contacts, such as in the design and optimization of electronic devices, electrical connectors, and power transmission systems. By elucidating the intricate interplay between electromagnetic fields and heat transfer phenomena, our work contributes to advancing the understanding and engineering of such systems in both industrial and academic contexts.

In essence, this research underscores the importance of integrating mathematical modeling with theoretical analysis to unravel the complexities inherent in multifaceted physical processes. Moving forward, we envision further refinement and extension of our models to address additional complexities and pave the way for enhanced performance and reliability of electrical contact systems in diverse real-world scenarios.

\section{Acknowledgement}
The study was sponsored by AP19675480 "Problems of heat conduction with a free boundary arising in modeling of switching processes in electrical devices" from Ministry of Sciences and Higher Education of the Republic of Kazakhstan and projects  O06-24CI1901 and  O06-24CI1903  from Austral University, Rosario

\end{document}